\documentclass[a4paper,11pt]{article}
\usepackage{amsmath}
\usepackage{amssymb}
\usepackage{theorem}
\usepackage{pstricks}
\usepackage{pstricks}
\usepackage{charter}
\usepackage{euscript}
\usepackage{epic,eepic}
\usepackage{graphicx}
\usepackage{ulem}
\usepackage[plc]{optional} 
\PassOptionsToPackage{normalem}{ulem}
\newcommand{\email}[1]{\href{mailto:#1}{\nolinkurl{#1}}}
\renewcommand{\leq}{\ensuremath{\leqslant}}
\renewcommand{\geq}{\ensuremath{\geqslant}}

\newcommand{\Frac}[2]{\displaystyle{\frac{#1}{#2}}}
\newcommand{\menge}[2]{\big\{{#1} \mid {#2}\big\}}
\newcommand{\Scal}[2]{{\bigg\langle{{#1}\:\bigg |~{#2}}\bigg\rangle}}
\newcommand{\Menge}[2]{\bigg\{{#1}~\bigg|~{#2}\bigg\}}
\newcommand{\lev}[1]{\ensuremath{\mathrm{lev}_{\leq #1}\:}}

\newcommand{\emp}{\ensuremath{{\varnothing}}}

\newcommand{\scal}[2]{\left\langle{#1}\mid {#2} \right\rangle}

\newcommand{\norm}[1]{\|#1\|}
\newcommand{\infconv}{\ensuremath{\mbox{\footnotesize$\,\square\,$}}}
\newcommand{\einfconv}{\ensuremath{\mbox{\footnotesize$\,\boxdot\,$}}}
\newcommand{\exi}{\ensuremath{\exists\,}}

\newcommand{\HH}{\ensuremath{\mathcal H}}
\newcommand{\GG}{\ensuremath{\mathcal G}}
\newcommand{\BL}{\ensuremath{\EuScript B}\,}

\newcommand{\RR}{\ensuremath{\mathbb R}}

\newcommand{\RX}{\ensuremath{\,\left]-\infty,+\infty\right]}}

\newcommand{\RXX}{\ensuremath{\left[-\infty,+\infty\right]}}
\newcommand{\NN}{\ensuremath{\mathbb N}}
\newcommand{\dom}{\ensuremath{\operatorname{dom}}}
\newcommand{\cont}{\ensuremath{\operatorname{cont}}}
\newcommand{\prox}{\ensuremath{\operatorname{prox}}}
\newcommand{\intdom}{\ensuremath{\operatorname{int}\operatorname{dom}}\,}
\newcommand{\inte}{\ensuremath{\operatorname{int}}}

\newcommand{\Argmin}{\ensuremath{\operatorname{Argmin}}}

\newcommand{\zer}{\ensuremath{\operatorname{zer}}}
\newcommand{\gra}{\ensuremath{\operatorname{gra}}}

\newcommand{\Fix}{\ensuremath{\operatorname{Fix}}}
\newcommand{\Id}{\ensuremath{\operatorname{Id}}}

\newcommand{\weakly}{\ensuremath{\rightharpoonup}}

\newtheorem{theorem}{Th\'eor\`eme}[section]
\newtheorem{lemma}[theorem]{Lemme}

\newtheorem{corollary}[theorem]{Corollaire}
\newtheorem{proposition}[theorem]{Proposition}
\newtheorem{definition}[theorem]{D\'efinition}
\theoremstyle{plain}{\theorembodyfont{\rmfamily}
}
\theoremstyle{plain}{\theorembodyfont{\rmfamily}
}
\theoremstyle{plain}{\theorembodyfont{\rmfamily}
}
\theoremstyle{plain}{\theorembodyfont{\rmfamily}
\newtheorem{example}[theorem]{Exemple}}
\theoremstyle{plain}{\theorembodyfont{\rmfamily}
}
\theoremstyle{plain}{\theorembodyfont{\rmfamily}
\newtheorem{remark}[theorem]{Remarque}}
\theoremstyle{plain}{\theorembodyfont{\rmfamily}
\newtheorem{notation}[theorem]{Notation}}

\numberwithin{equation}{section}

\usepackage[french]{babel}

\newcommand{\Diff}{\ensuremath{\operatorname{D}}}
\newcommand{\Inv}{\ensuremath{\operatorname{Inv}}}
\newcommand{\Li}{\ensuremath{\underset{\mathrm{P-K}}{\operatorname{liminf}}}}
\newcommand{\Ls}{\ensuremath{\underset{\mathrm{P-K}}{\operatorname{limsup}}}}

\newcommand{\linf}{\ensuremath{\underline{\operatorname{lim}}}}
\newcommand{\lsup}{\ensuremath{\overline{\operatorname{lim}}}}
\newcommand{\cl}[1]{\ensuremath{\overline{#1}}}
\newcommand{\epi}{\ensuremath{\operatorname{epi}}}
\newcommand{\clev}[1]{\ensuremath{\mathrm{lev}_{> #1}\:}}
\newcommand{\im}{\ensuremath{\operatorname{im}}}
\newcommand{\fr}{\ensuremath{\operatorname{fr}}}

\newcommand{\Hfo}{\ensuremath{\mathcal H}^{\textup{fort}}}
\newcommand{\Hfa}{\ensuremath{\mathcal H}^{\textup{faible}}}

\newcommand{\pk}{\overset{\textsc{P-K}}{\longrightarrow}}

\newcommand{\dueto}[2]{\textup{\cite[#1]{#2}}}

\begin{document}


\begin{center}
{\sffamily{\Large\textbf{-- Master's Degree Memoir --}}}

\vskip 12mm

{\sffamily{\Large\textbf{Subgradient Projection Operators}}}

\vskip 12mm

{\sffamily{\Large\textbf{ by Beno\^it Pauwels}

\vskip 12mm

{\sffamily{\large under the supervision of :\\[4mm]
Patrick L. Combettes \\[4mm]
\small UPMC Universit\'e Paris 06\\
\small Laboratoire Jacques-Louis Lions-- UMR CNRS 7598\\
\small 75005 Paris, France\\
\small\ttfamily{plc@math.jussieu.fr} }}
}}
\vskip 11mm

\end{center}

\noindent
{\bfseries Abstract.}
Several algebraic and topological properties of subgradient 
projection operators are investigated and various examples are 
provided. Connections with
Moreau's proximity operator are also made and acceleration schemes
for subgradient projection algorihms are discussed. Finally
continuity, nonexpansiveness, monotonicity, differentiability, and 
epi-convergence properties are investigated.

\vfill

\begin{center}
{\normalsize September 20, 2012}

{\normalsize Laboratoire Jacques-Louis Lions --
Universit\'e Pierre et Marie Curie\\}
\end{center}

\newpage
{\centering
{\sffamily{\Large\textbf{-- M\'emoire de Master OJME --}}}
\vfill
{\sffamily{\Large\textbf{Op\'erateurs de projection 
sous-diff\'erentielle}}}
\vfill
{\sffamily{\Large par Beno\^it Pauwels}}
\vfill
{\sffamily{\Large sous la direction de :\\[4mm]
Patrick L. Combettes \\[4mm]
\small UPMC Universit\'e Paris 06\\
\small Laboratoire Jacques-Louis Lions-- UMR CNRS 7598\\
\small 75005 Paris, France\\
\small\ttfamily{plc@math.jussieu.fr} }}
\vfill
{\normalsize 20 septembre 2012}
\vfill
{\normalsize Laboratoire Jacques-Louis Lions --
Universit\'e Pierre et Marie Curie\\}
}


\newpage

\tableofcontents


\newpage

\section*{Introduction}

La notion de {\em projecteur sous-diff\'erentiel} a \'et\'e 
introduite par Naum Z. Shor dans un algorithme de r\'esolution 
de programmes lin\'eaires \cite{Shor62}. Elle a \'et\'e approfondie 
et d\'evelopp\'ee par le m\^eme Shor \cite{Shor64,Shor85,Shor98}, 
Yuri M. Ermoliev \cite{Ermo,ErSh} et Boris T. Polyak 
\cite{Poly67,Poly78,Poly87} notamment.
\newline
Dans toute la suite $\HH$ d\'esignera un espace de Hilbert sur 
$\RR$. Nous allons nous int\'eresser au probl\`eme suivant: 
\'etant donn\'es un point $x_0\in\HH$ et une fonction de 
$\Gamma_0(\HH)$ - c'est-\`a-dire une fonction de $\HH$ dans $\RX$ 
propre convexe semi-continue inf\'erieurement - on cherche la 
projection m\'etrique de $x_0$ sur 
$\lev{0}f=\menge{x\in\HH}{f(x)\leq0}$ (suppos\'e non vide).
\newline
Voici une version de l'algorithme de Shor, \'etudi\'ee par Polyak 
\cite{Poly87}:
\begin{equation}
\label{algo hist}
x_{n+1}=x_n-\Frac{f(x_n)}{\norm{u_n}^2}u_n
\end{equation}
o\`u $(\forall n\in\NN)~u_n\in\partial f(x_n)$. Sous certaines 
conditions les orbites cet algorithme convergent faiblement vers 
des points de $\lev{0}f$ \cite{Mor01}.
\newline
Pour toute s\'election $U$ de $\partial f$ et tout point 
$x\in\dom U$ nous appellerons {\em projection 
sous-diff\'erentielle} 
de $x$ dans la direction $Ux$ le point 
\begin{equation}
G_f^Ux=x-\Frac{f(x)}{\norm{Ux}^2}Ux.
\end{equation}
Plus g\'en\'eralement on 
s'int\'eressera \`a l'ensemble 
\begin{equation}
G_fx=\Menge{x-\Frac{f(x)}{\norm{u}^2}u}{u\in\partial f(x)}.
\end{equation}
L'algorithme (\ref{algo hist}) se r\'e\'ecrit alors simplement
\begin{equation}
x_{n+1}=G_f^Ux_n
\end{equation}
o\`u $U$ est une s\'election qui prolonge la suite $(u_n)_{n\in\NN}$.
\newline
Dans ce document nous \'etudierons l'op\'erateur de 
projection sous-diff\'erentielle et ses s\'elections. Nous nous 
int\'eresserons notamment \`a leurs propri\'et\'es alg\'ebriques 
(projecteur sous-diff\'erentiel d'une composition, d'une combinaison 
affine, d'une inf-convolution, etc.), \`a leurs liens avec la 
conjugu\'ee de Fenchel et l'enveloppe de Moreau, \`a leur 
r\'egularit\'e (continuit\'e, diff\'erentiabiliti\'e, caract\`ere 
lipschitzien), au comportement de la suite des projecteurs 
sous-diff\'erentiels 
associ\'ee \`a une suite de fonctions de 
$\Gamma_0(\HH)$, et enfin aux propri\'et\'es d'op\'erateur 
multivoque de $G_f$ (propri\'et\'es des valeurs, semi-continuit\'e 
et monotonie).

\section*{Contributions}

\`A notre connaissance l'\'etude de la projection 
sous-diff\'erentielle 
comme op\'erateur est nouvelle, et l'ensemble des 
propri\'et\'es alg\'ebriques, topologiques et s\'equentielles 
d\'emontr\'ees ici sur cet objet sont originales.


\newpage


Dans toute ce document $\HH$ d\'esigne un espace de Hilbert sur 
$\RR$.

\section{Notations, d\'efinitions et exemples}

\subsection{Projecteur sous-diff\'erentiel}

\'Etant donn\'ee $f\in\Gamma_0(\HH)$ telle que 
$\lev{0}f=\menge{x\in\HH}{f(x)\leq0}\neq\emp$ on veut 
approcher l'op\'erateur de projection m\'etrique sur 
$\lev{0}f$ par un op\'erateur plus 
"simple" \`a calculer. L'id\'ee de la projection 
sous-diff\'erentielle 
est de projeter sur un demi-espace contenant 
$\lev{0}f$ d\'efini \`a l'aide d'un sous-gradient de $f$. 
Cette approximation est particuli\`erement utile dans de 
nombreux probl\`emes d'optimisation 
\cite{Baus96,Imag97,Else01}.

\begin{notation}
Soient $f\in\Gamma_0(\HH)$ telle que $\lev{0}f\neq\emp$ et 
$(x,u)\in\gra\partial f$. On d\'efinit 
l'approximation suivante de $\lev{0}f$:
\begin{equation}
H_u=\menge{y\in\HH}{\scal{y-x}{u}+f(x)\leq0}
\end{equation} 
\end{notation}

\begin{remark}
$H_u$ est un convexe ferm\'e contenant $\lev{0}f$ 
(donc non vide). On 
observe que: $x\in H_u\Leftrightarrow f(x)\leq0$.
\end{remark}

Le lemme suivant, qui est une cons\'equence directe de la r\`egle 
de Fermat, affirme que les sous-gradients des points de 
$\clev{0}f=\menge{x\in\HH}{f(x)>0}$ sont non nuls.

\begin{lemma}
\label{sous-gradient non nul}
Soit $f\in\Gamma_0(\HH)$ telle que $\lev{0}f\neq\emp$.
Alors $\zer\partial f\subset\lev{0}f$.
\end{lemma}

\begin{lemma}
\dueto{Example 28.16}{Livre1}
\label{proj metrique hyperplan}
Soient $f\in\Gamma_0(\HH)$ telle que $\lev{0}f\neq\emp$ 
et $(x,u)\in\gra\partial f$.
L'expression de $P_{H_u}x$ est donn\'ee par
\begin{equation}
P_{H_u}x=
\begin{cases}
x & \text{si }f(x)\leq0\\
x-\Frac{f(x)}{\norm{u}^2}u & \text{si }f(x)>0
\end{cases}
\end{equation}
\end{lemma}

\begin{definition}
Soit $f\in\Gamma_0(\HH)$ telle que $\lev{0}f\neq\emp$.
On appellera {\em projecteur sous-diff\'erentiel} associ\'e \`a $f$ 
l'op\'erateur multivoque suivant :
\begin{equation}
G_f:\HH\rightarrow2^{\HH}:x\mapsto
\begin{cases}
\{x\} & \text{si }f(x)\leq0\\
\Menge{x-\Frac{f(x)}{\norm{u}^2}u}{u\in\partial f(x)}
& \text{si }f(x)>0.
\end{cases}
\end{equation}
\end{definition}

Observons que $\left(\forall x\in\dom\partial f\right)~G_fx=
\menge{P_{H_u}x}{u\in\partial f(x)}$. Mais l'exemple suivant 
montre qu'il peut exister des points $x\in\lev{0}f$ tels que 
$\partial f(x)=\emp$.

\begin{example}
\dueto{Example 3.8(a)}{Phel93}
Consid\'erons que $\HH$ est l'espace des suites r\'eelles de 
carr\'es sommables et notons 
$C=\menge{x\in\HH}{(\forall n\in\NN)~|x_n|\leq2^{-n}}$. Alors la 
fonction
\begin{equation}
f:\HH\rightarrow\RXX:x\mapsto
\begin{cases}
\underset{n\in\NN}{\sum}\left[-\left(2^{-n}+x_n\right)^{1/2}\right] 
& \text{si }x\in C\\
+\infty & \text{sinon}
\end{cases}
\end{equation}
est convexe semi-continue inf\'erieurement et tout \'el\'ement 
$x$ de $C$ tel que $x_n>-2^{-n}$ pour une infinit\'e d'indices 
v\'erifie $\partial f(x)=\emp$. Par exemple $\partial f(0)=\emp$ 
et $f(0)=-\Frac{\sqrt{2}}{\sqrt{2}-1}<0$, d'o\`u 
$x\in\lev{0}f\cap\complement\dom\partial f$.
\end{example}

\begin{remark}\
\label{projections}
Soit $f\in\Gamma_0(\HH)$ telle que $\lev{0}f\neq\emp$.
\begin{enumerate}
\item
$\dom G_f=\lev{0}f\cup\dom\partial f\subset\dom f$
\item
\label{lien proj metrique}
$\left(\forall x\in\dom\partial f\right)\quad G_fx=
\menge{P_{H_u}x}{u\in\partial f(x)}$
\item
$\Fix G_f=\lev{0}f$
\item
$\dom G_f=\HH\Leftrightarrow\clev{0}f\subset\dom\partial f
\Leftrightarrow\dom f=\HH$. Dans ce cas $\cont f=\dom\partial f=
\dom f=\HH$.
\item
Si $f$ est G\^ateaux-diff\'erentiable sur $\clev{0}f$ 
alors $G_f$ est univoque :
\begin{equation}
G_f:\HH\rightarrow\HH:x\mapsto
\begin{cases}
x & \text{si }f(x)\leq0\\
x-\Frac{f(x)}{\norm{\nabla f(x)}^2}\nabla f(x) & \text{si }f(x)>0.
\end{cases}
\end{equation}
\item
Supposons $\HH=\RR$ et $f$ diff\'erentiable.
Alors $\dom G_f=\RR$ et
\begin{equation}
G_f:\RR\rightarrow\RR:x\mapsto
\begin{cases}
x & \text{si }f(x)\leq0\\
x-\Frac{f(x)}{f'(x)} & \text{si }f(x)>0.
\end{cases}
\end{equation}
\item
$G_f=\Id\Leftrightarrow(\exi\eta\in\RR_{-})~f\equiv\eta$. En effet 
toute 
fonction convexe major\'ee sur $\HH$ est constante.
\end{enumerate}
\end{remark}

\begin{proposition}
\label{quasi-contractance}
Soit $f\in\Gamma_0(\HH)$ telle que $\lev{0}f\neq\emp$. Alors
\begin{equation}
\left(\forall (x,p)\in\gra G_f\right)\left(\forall y\in
\lev{0}f\right)~\norm{p-y}^2\leq\norm{x-y}^2-\norm{x-p}^2.
\end{equation}
En particulier $G_f$ est une quasi-contraction :
\begin{equation}
\left(\forall (x,p)\in\gra G_f\right)\left(\forall y\in
\lev{0}f\right)~\norm{p-y}\leq\norm{x-y}.
\end{equation}
\end{proposition}

\begin{proof}
Soient $(x,p)\in\gra G_f$ et $y\in\lev{0}f$. Si $f(x)\leq0$ alors 
$p=x$, l'in\'egalit\'e est donc v\'erifi\'ee. Supposons $f(x)>0$. 
Il existe alors $u\in\partial f(x)$ tel que $p=P_{H_u}x$ 
(Remarque \ref{projections}\ref{lien proj metrique}). Par 
cons\'equent
\begin{align}
\norm{p-y}^2+\norm{x-p}^2
&=\norm{x-y}^2-2\scal{p-y}{x-p}\nonumber\\
&=\norm{x-y}^2+2\scal{y-P_{H_u}x}{x-P_{H_u}x}\nonumber\\
&\leq\norm{x-y}^2.
\end{align}
\end{proof}

L'exemple suivant montre que la notion de projecteur 
sous-diff\'erentiel 
\'etend celle de projecteur m\'etrique.

\begin{example}
\label{proj distance}
Soit $C\subset\HH$ un convexe ferm\'e non 
vide. Alors
\begin{equation}
G_{d_C}=P_C.
\end{equation}
En particulier, pour $C=\{0\}$,
\begin{equation}
G_{\norm{\cdot}}\equiv0.
\end{equation}
\end{example}

\begin{proof}
On a $\lev{0}d_C=C$ et $(\forall x\in\complement C)~\nabla
d_C(x)=\Frac{x-P_Cx}{d_C(x)}$ 
\dueto{Proposition 18.22(iii)}{Livre1}. Par cons\'equent 
$\dom G_{d_C}=\HH$ 
et $(\forall x\in\complement C)~G_{d_C}x=P_Cx$.
\end{proof}

\subsection{S\'elections du projecteur sous-diff\'erentiel}

Dans la litt\'erature les projections sous-diff\'erentielles 
apparaissent sous forme de s\'elections de l'op\'erateur 
mutlivoque que nous venons de d\'efinir : \`a chaque \'etape de 
l'algorithme on se contente d'un seul sous-gradient du point 
courant.

\subsubsection{S\'elections}

\begin{definition}
Soit $f\in\Gamma_0(\HH)$.
On appellera {\em s\'election} de $\partial f$ tout
op\'erateur $U$ d\'efini sur une partie $D$ de $\dom
\partial f$ tel que $(\forall x\in D)~Ux\in\partial f(x)$.
\end{definition}

\begin{notation}
Soient $f\in\Gamma_0(\HH)$ telle que $\lev{0}f\neq\emp$ et $U$ 
une s\'election de $\partial f$. 
On remarque tout d'abord que, d'apr\`es le Lemme 
\ref{sous-gradient non nul}, $\left(\forall x\in\clev{0}f\cap
\dom U\right)~Ux\neq0$. 
On notera $G_f^U$ l'op\'erateur univoque d\'efini sur 
$\lev{0}f\cup\dom U$ de la mani\`ere suivante:
\begin{equation}
G_f^Ux=
\begin{cases}
x & \text{si }f(x)\leq0\\
x-\Frac{f(x)}{\norm{Ux}^2}Ux & \text{si }f(x)>0\text{ et }
x\in\dom U
\end{cases}
\end{equation}
\end{notation}

Dans la suite il arrivera que l'on ait \`a consid\'erer des 
op\'erateurs $G_f^U$ o\`u $\dom U$ est un sous-ensemble propre de 
$\dom\partial f$, d'o\`u la d\'efinition assez large
de {\em{s\'election}} donn\'ee ci-dessus.

\begin{remark}
\label{premieres proprietes}
Soient $f\in\Gamma_0(\HH)$ telle que $\lev{0}f\neq\emp$ et $U$ 
une s\'election de $\partial f$.
\begin{enumerate}
\item
$\dom G_f^U\subset\dom f$
\item
\label{selproj et proj}
$\left(\forall x\in\dom U\right)\quad G_f^Ux=P_{H_{Ux}}x$
\item
$\left(\forall x\in\dom G_f^U\right)\left(\forall y\in\lev{0}f
\right)\quad\Scal{y-G_f^Ux}{x-G_f^Ux}\leq0$
\item
$\left(\forall x\in\dom G_f^U\right)\quad G_f^Ux=P_{\lev{0}f}x
\Leftrightarrow G_f^Ux\in\lev{0}f$
\item
$\Fix G_f^U=\lev{0}f$
\end{enumerate}
\end{remark}

\subsubsection{Classe $\mathfrak{T}$}

Nous allons voir que les s\'elections de projecteurs 
sous-diff\'erentiels 
d\'efinies sur tout $\HH$ appartiennent \`a 
une classe d'op\'erateurs d\'ej\`a \'etudi\'ee dans 
\cite{Mor01,Else01}.

\begin{notation}
\dueto{(1.2)}{Mor01}
Soit $(x,y)\in\HH^2$. Notons
\begin{equation}
H(x,y)=\menge{u\in\HH}{\scal{u-y}{x-y}\leq0}.
\end{equation}
On observe que $y=P_{H(x,y)}x$.
\end{notation}

\begin{definition}
\dueto{Definition 2.2.}{Mor01}
On d\'efinit la classe d'op\'erateurs suivante:
\begin{equation}
\mathfrak{T}=\menge{T:\HH\rightarrow\HH}{(\forall x\in\HH)~\Fix T
\subset H(x,Tx)}
\end{equation}
\end{definition}

\begin{remark}
\dueto{Proposition 2.3}{Mor01}
Soit $T:\HH\rightarrow\HH$. Les assertions suivantes sont 
\'equivalentes.
\begin{enumerate}
\item
$T\in\mathfrak{T}$
\item
$(\forall x\in\HH)(\forall y\in\Fix T)\quad\scal{y-Tx}{x-Tx}\leq0$
\item
$2T-\Id$ est une quasi-contraction.
\end{enumerate}
\end{remark}

\begin{notation}
Notons $\mathfrak{G}$ la classe des s\'elections de projecteurs 
sous-diff\'erentiels d\'efinies sur $\HH$:
\begin{equation}
\mathfrak{G}=\menge{G:\HH\rightarrow\HH}{(\exi f\in\Gamma_0(\HH))
(\exi U\text{ s\'election de }
\partial f)~G=G_f^U\text{ et }\clev{0}f\subset\dom U}.
\end{equation}
\end{notation}

\begin{remark}
Soient $f\in\Gamma_0(\HH)$ et $U$ une s\'election de $\partial f$. 
Si $G_f^U\in\mathfrak{G}$ alors 
$\cont f=\dom\partial f=\dom f=\HH$.
\end{remark}

\begin{proposition}
\dueto{Proposition 2.3.}{Mor01}
\begin{equation}
\mathfrak{G}\subset\mathfrak{T}
\end{equation}
\end{proposition}

\begin{proposition}
\dueto{Proposition 2.3}{Else01}
Soit $T\in\mathfrak{T}$.
\begin{enumerate}
\item
$\left(\forall(x,y)\in\HH\times\Fix T\right)\quad\norm{Tx-x}^2
\leq\scal{y-x}{Tx-x}$
\item
$(\forall\alpha\in[0,2])\left(\forall(x,y)\in\HH\times\Fix T\right)
\quad\norm{(1-\alpha)x+\alpha Tx-y}^2\leq\norm{x-y}^2-\alpha(2-
\alpha)\norm{Tx-x}^2$. En particulier $T$ est une quasi-contraction.
\item
$(\forall x\in\HH)\quad\norm{Tx-x}\leq d_{\Fix T}x$
\item
$\Fix T=\underset{x\in\HH}{\bigcap}H(x,Tx)$
\item
$\Fix T$ est convexe ferm\'e. D'o\`u:
\begin{equation}
(\forall x\in\HH)\quad Tx\in\Fix T\Leftrightarrow Tx=P_{\Fix T}x.
\end{equation}
\item
$(\forall\alpha\in[0,1])\quad(1-\alpha)\Id+\alpha T\in\mathfrak{T}$
\end{enumerate}
\end{proposition}

\begin{remark}
Soient $f\in\Gamma_0(\HH)$ et $U$ une s\'election de $\partial f$. 
On suppose $G_f^U\in\mathfrak{G}$.
\begin{equation}
(\forall x\in\HH)\quad H\left(x,G_f^Ux\right)=
\begin{cases}
\HH & \text{si }f(x)\leq0\\
H_{Ux} & \text{si }f(x)>0.
\end{cases}
\end{equation}
On retrouve le point \ref{selproj et proj} de la Remarque 
\ref{premieres proprietes}: 
$(\forall x\in\dom U)~G_f^Ux=P_{H_{Ux}}x$.
\end{remark}

\begin{proof}
Soient $x\in\clev{0}f$ et $y\in\HH$. Alors
\begin{align}
\Scal{y-G_f^Ux}{x-G_f^Ux}
&=\Scal{y-x+\Frac{f(x)}{\norm{Ux}^2}Ux}{\Frac{f(x)}{\norm{Ux}^2}Ux}
\nonumber\\
&=\Frac{f(x)}{\norm{Ux}^2}(\scal{y-x}{Ux}+f(x)).
\end{align}
\end{proof}

\begin{corollary}
\label{prop heritees de T}
Soient $f\in\Gamma_0(\HH)$ et $U$ une s\'election de $\partial f$. 
On suppose $G_f^U\in\mathfrak{G}$.
\begin{enumerate}
\item
$G_f^U$ est une quasi-contraction. (Nous l'avions d\'ej\`a montr\'e 
\`a la Proposition \ref{quasi-contractance}.)
\item
$\left(\forall x\in\clev{0}f\right)\quad\Frac{f(x)}{\norm{Ux}}\leq
d_{\lev{0}f}(x)$
\item
\label{ineg dist}
$\lev{0}f=\underset{x\in\clev{0}f}{\bigcap}H_{Ux}$
\item
\label{moyenne T}
$(\forall\alpha\in[0,1])\quad(1-\alpha)\Id+\alpha G_f^U\in
\mathfrak{T}$
\end{enumerate}
\end{corollary}

\subsection{Exemples}

\begin{example}
Soit $u\in\HH$. On consid\`ere la fonction $f:\HH\rightarrow\HH:x
\mapsto\scal{x}{u}$. Alors
\begin{equation}
G_f:\HH\rightarrow\HH:x\mapsto
\begin{cases}
x & \text{si }\scal{x}{u}\leq0\\
x-\Frac{\scal{x}{u}}{\norm{u}^2}u & \text{si }\scal{x}{u}>0
\end{cases}.
\end{equation}
D'apr\`es \dueto{Example 28.16}{Livre1}, $G_f$ est le projecteur 
sur $\lev{0}f$:
\begin{equation}
G_{\scal{\cdot}{u}}=P_{\lev{0}f}.
\end{equation}
\end{example}

\begin{example}
\label{proj distance carree}
Soit $C\subset\HH$ un convexe ferm\'e non vide. 
Alors $\lev{0}d_C^2=C$ et $\nabla d_C^2=2(\Id-P_C)$ 
\dueto{Corollary 12.30}{Livre1}. On en d\'eduit 
que
\begin{equation}
G_{d_C^2}=\Frac{\Id+P_C}{2}.
\end{equation}
On remarque que, comme $P_C$ est une contraction, $G_{d_C^2}$ est 
une contraction ferme \dueto{Proposition 4.2}{Livre1}.
En particulier, pour $C=\{0\}$,
\begin{equation}
G_{\norm{\cdot}^2}=\Frac{\Id}{2}.
\end{equation}
\end{example}

\begin{example}
\dueto{Remark 2.4}{Mor01}
On pose $f:\RR\rightarrow\RR:x\mapsto\max\{x+1,
2x+1\}$. Alors 
\begin{equation}
G_f:\RR\rightarrow\RR:x\mapsto
\begin{cases}
x & \text{si }x\leq-1\\
-1 & \text{si }-1<x<0\\
\left[-1,-\frac{1}{2}\right] & \text{si }x=0\\
-\Frac{1}{2} & \text{si }x>0
\end{cases}.
\end{equation}
$f$ est convexe continue sur $\RR$ et diff\'erentiable sur $\RR
\backslash\{0\}$ mais aucune s\'election de $G_f$ n'est continue
 en $0$.
\end{example}

\begin{example}
Soit $\eta\in\RR_{++}$. On pose
\begin{equation}
f:\RR\rightarrow\RX:x\mapsto
\begin{cases}
\eta-\sqrt{x} & \text{si }x>0\\
+\infty & \text{si }x\leq0.
\end{cases}
\end{equation}
Alors $\lev{0}f=\left[\eta^2,+\infty\right[$, 
$\dom G_f=\RR_{++}$ et
\begin{equation}
G_f:\RR_{++}\rightarrow\RR:x\mapsto\
\begin{cases}
x & \text{si }x\geq\eta^2\\
2\eta\sqrt{x}-x & \text{si }0<x<\eta^2
\end{cases}.
\end{equation}
$G_f$ est continu sur $\RR_{++}$ mais pas lipschitzien.
\end{example}

\begin{example}
Soit $\eta\in\left]1,+\infty\right[$. On pose 
$f:\RR\rightarrow\RR:x\mapsto
\sqrt{1+x^2}-\eta$. Alors
\begin{equation}
G_fx=
\begin{cases}
x & \text{si }|x|\leq\sqrt{\eta^2-1}\\
\Frac{x^2-1+\eta\sqrt{1+x^2}}{2x} & \text{si }|x|>
\sqrt{\eta^2-1}.
\end{cases}
\end{equation}
De plus
\begin{equation}
(\forall x\in\RR)~G_f'(x)=
\begin{cases}
1 & \text{si }|x|\leq\sqrt{\eta^2-1}\\
\Frac{3}{2}+\Frac{\eta}{\sqrt{1+x^2}}\left(1+\Frac{1}{x^2}\right)-
\Frac{1}{2x^2}-\Frac{\eta}{2x^2\sqrt{1+x^2}} & \text{si }|x|>
\sqrt{\eta^2-1}.
\end{cases}
\end{equation}
Pour tout $x\in\RR$ tel que $|x|>\sqrt{\eta^2-1}$, on a 
$\Frac{3}{2}-\Frac{1}{\eta^2-1}\leq
G_f'(x)\leq\Frac{5}{2}+\Frac{1}{\eta^2-1}$. $G_f'$ est donc 
born\'ee. Ainsi $f$ est convexe diff\'erentiable lipschitzienne 
et $G_f$ est lipschitzien.
\end{example}

\begin{example}
\label{ex log}
On pose
\begin{equation}
f:\RR\rightarrow\RX:x\mapsto
\begin{cases}
-\operatorname{ln}(x) & \text{si }x>0\\
+\infty & \text{sinon.}
\end{cases}
\end{equation}
Alors $\lev{0}f=[1,+\infty[$ et
\begin{equation}
(\forall x\in\RR_{++})\quad G_fx=
\begin{cases}
x & \text{si }x\geq1\\
x-x\operatorname{ln}(x) & \text{si }0<x<1.
\end{cases}
\end{equation}
\end{example}

\subsection{R\'e\'ecriture de l'algorithme de projection 
sous-diff\'erentielle}

Rappelons la m\'ethode de projection sous-diff\'erentielle.

\begin{proposition}
\dueto{Corollary 6.10}{Mor01}
Soient $(f_i)_{i\in I}$ une famille d\'enombrable de fonctions 
convexes continues de $\HH$ dans $\RR$ telle que 
$(\forall i\in I)~\lev{0}f_i\neq\emp$, $U_i$ des s\'elections de 
$\partial f_i$ (pour $i\in I$), $x_0\in\HH$ et 
$\varepsilon\in]0,1]$. On suppose que l'injection 
$i:\NN\rightarrow I$ v\'erifie
\begin{equation}
(\forall i\in I)(\exi M_i\in\NN\backslash\{0\})(\forall n\in\NN)
\quad i\in\{i(n),...i(n+M_i-1)\},
\end{equation}
que, pour tout $i$ dans $I$, $\partial f_i$ est born\'e sur les 
born\'es et que $S=\underset{i\in I}{\bigcap}\lev{0}f_i\neq\emp$. 
Alors chaque orbite de l'algorithme
\begin{equation}
x_{n+1}=x_n+\lambda_n\left(G_{f_{i(n)}}^{U_{i(n)}}x_n-x_n\right)
\text{ o\`u }\lambda_n\in[\varepsilon,2-\varepsilon]
\end{equation}
converge faiblement vers un point de $S$.
\end{proposition}

\section{Propri\'et\'es alg\'ebriques}

\'Etudions \`a pr\'esent les propri\'et\'es alg\'ebriques du 
projecteur sous-diff\'erentiel, \textit{i.e.} comment s'exprime 
le projecteur sous-diff\'erentiel d'une compos\'ee, d'une somme, 
d'une inf-convolution, etc.

\subsection{Propri\'et\'es li\'ees \`a la composition}

\begin{proposition}
\label{proj produit externe}
Soient $f\in\Gamma_0(\HH)$ telle que $\lev{0}f\neq\emp$, 
$\lambda\in\RR_{++}$ et $U$ une s\'election de 
$\partial f$. Alors $G_{\lambda f}$ et 
$G_{\lambda f}^{\lambda U}$ sont bien d\'efinis, et
\begin{equation}
G_{\lambda f}=G_f
\end{equation}
\begin{equation}
G_{\lambda f}^{\lambda U}=G_f^U
\end{equation}
\end{proposition}

\begin{proof}
Tout d'abord $G_{\lambda f}$ est bien d\'efini puisque $\lambda 
f\in\Gamma_0(\HH)$ et $\lev{0}(\lambda f)=\lev{0}f\neq\emp$. Comme 
$\partial(\lambda f)=\lambda\partial f$, d'une
part $\dom G_{\lambda f}=\dom G_f$, d'autre part, pour toute
s\'election $U$ de $\partial f$, $\dom(\lambda U)=\dom U$ et 
$\lambda U$ est une s\'election de $\partial(\lambda f)$,
$G_{\lambda f}^{\lambda U}$ est donc bien d\'efini. Enfin
$(\forall x\in\clev{0}f\cap\dom U)~G_{\lambda f}^{\lambda U}x
=x-\Frac{\lambda f(x)}{\norm{\lambda Ux}^2}\lambda Ux
=x-\Frac{f(x)}{\norm{Ux}^2}Ux=G_f^Ux$.
\end{proof}

\begin{proposition}
\label{proj composition gauche}
Soient $f\in\Gamma_0(\HH)$ \`a valeurs r\'eelles telle que $\lev{0}f
\neq\emp$ et $\phi:\RR\rightarrow\RR$ strictement croissante sur 
$f(\HH)$ avec $\phi(0)=0$. On suppose que $f$ est 
Fr\'echet-diff\'erentiable sur $\clev{0}f$ et que $\phi$ est 
diff\'erentiable sur $\RR$. Alors $\dom G_{\phi\circ f}=\HH$ et
\begin{equation}
G_{\phi\circ f}:\HH\rightarrow\HH:x\mapsto
\begin{cases}
x & \text{si }f(x)\leq0\\
x+\Frac{\phi(f(x))}{f(x)\phi'(f(x))}\left(G_fx-x\right) & 
\text{si }f(x)>0
\end{cases}.
\end{equation}
\end{proposition}

\begin{proof}
Soit $x\in\clev{0}f$. Alors $\nabla(\phi\circ f)(x)=
\phi'(f(x))\nabla f(x)$, d'o\`u
\begin{align}
G_{\phi\circ f}x
&=x-\Frac
{\phi(f(x))}{\phi'(f(x))f(x)}\Frac{f(x)}{\norm{\nabla f(x)}^2}
\nabla f(x)\nonumber\\
&=x+\Frac{\phi(f(x))}{f(x)\phi'(f(x))}\left(G_fx-x
\right)
\end{align}
\end{proof}

\begin{corollary}
\label{proj puissance}
Soient $f\in\Gamma_0(\HH)$ \`a valeurs r\'eelles telle que $\lev{0}f
\neq\emp$ et $\alpha\in\RR_{++}$. On suppose $f\geq0$ et $f$ 
Fr\'echet-diff\'erentiable sur $\clev{0}f$. Alors
\begin{equation}
G_{f^{1/\alpha}}=(1-\alpha)\Id+\alpha G_f.
\end{equation}
\end{corollary}

\begin{proof}
Prendre $\phi:\RR\rightarrow\RR:t\mapsto |t|^{1/\alpha}$ dans la 
Proposition
\ref{proj composition gauche}.
\end{proof}

\begin{remark}
Ce r\'esultat est \`a comparer avec le Corollaire 
\ref{prop heritees de T}\ref{moyenne T}.
\end{remark}

\begin{corollary}
\label{proj puissance distance}
Soient $C$ une partie convexe ferm\'ee non vide de $\HH$ et 
$\alpha\in\RR_{++}$. Alors
\begin{equation}
G_{d_C^{1/\alpha}}=(1-\alpha)\Id+\alpha P_C.
\end{equation}
En particulier, si $\alpha\in]0,1]$ alors $G_{d_C^{1/\alpha}}$ est 
$\alpha$-moyenn\'e, donc est une contraction 
\dueto{Remark 4.24 (i)}{Livre1}, et m\^eme une contraction ferme 
quand $\alpha\in\left]0,1/2\right]\cup\{1\}$ 
\dueto{Remark 4.27}{Livre1}.
\end{corollary}

\begin{proof}
Prendre $f=d_C$ dans le Corollaire \ref{proj puissance}.
\end{proof}

\begin{remark}
On retrouve Exemple \ref{proj distance} et Exemple 
\ref{proj distance carree} pour les valeurs $\alpha=1$ et 
$\alpha=1/2$ respectivement.
\end{remark}

\begin{corollary}
\label{proj puissance norme}
Soit $\alpha\in\RR_{++}$. Alors
\begin{equation}
G_{\norm{\cdot}^{1/\alpha}}=(1-\alpha)\Id.
\end{equation}
\end{corollary}

\begin{proof}
Cela r\'esulte du Corollaire \ref{proj puissance distance} o\`u 
l'on a pris $C=\{0\}$.
\end{proof}

\begin{proposition}
\dueto{Proposition 16.5}{Livre1}
\label{sousdiff composition}
Soient $f\in\Gamma_0(\HH)$, $\GG$ un espace de Hilbert sur $\RR$ et 
$L\in\mathcal{B}(\GG,\HH)$. Alors $L^{\ast}\circ\partial f
\circ L\subset\partial(f\circ L)$.
\end{proposition}

\begin{proposition}
\label{proj composition droite}
Soient $f\in\Gamma_0(\HH)$ telle que $\lev{0}f\neq\emp$, $\GG$ un 
espace de Hilbert sur $\RR$, $L\in\mathcal{B}(\GG,\HH)$ et $U$ une 
s\'election de $\partial f$. On suppose $L^{-1}(\lev{0}f)\neq\emp$.
\begin{enumerate}
\item
$G_{f\circ L}^{L^{\ast}UL}$ est bien d\'efini.
\item
$\dom G_{f\circ L}^{L^{\ast}UL}=L^{-1}\left(
\dom G_f^U\right)$
\item
$\left(\forall y\in\dom G^{L^{\ast}UL}
_{f\circ L}\right)\quad\left[L^{\ast}\circ G_f^U\circ L
\right]y=L^{\ast}Ly-\Frac{\norm{L^{\ast}ULy}^2}{\norm{
ULy}^2}\left(y-G_{f\circ L}^{L^{\ast}UL}y\right)$
\item
S'il existe $\alpha\in\RR$ tel que $L^\ast L=LL^\ast=\alpha\Id$ 
alors $\alpha G_{f\circ L}^{L^\ast UL}=L^\ast\circ G_f^U\circ L$.
\item
Si $L$ est inversible avec $L^{-1}=L^{\ast}$ alors $G^{L^{\ast}UL}
_{f\circ L}=L^{\ast}\circ G_f^U\circ L$.
\end{enumerate}
\end{proposition}

\begin{proof}
\begin{enumerate}
\item
D'une part $\lev{0}(f\circ L)=L^{-1}(\lev{0}f)\neq\emp$, 
d'autre part $(\forall y\in\dom U)~L^{\ast}ULy\in\partial(f\circ L)
(y)$ d'apr\`es la Proposition \ref{sousdiff composition}.
\item
On a $\dom G_{f\circ L}^{L^{\ast}UL}
=L^{-1}(\lev{0}f)\cup L^{-1}(\dom U)
=L^{-1}(\dom G_f^U)$.
\item
Soit $y\in\dom G^{L^{\ast}UL}_{f\circ L}$. 
Si $f(Ly)\leq0$ alors $[L^{\ast}\circ G_f^U\circ L]y
=L^{\ast}Ly$. Si $f(Ly)>0$ alors
\begin{align}
[L^{\ast}\circ G_f^U\circ L]y
&=L^{\ast}Ly-\Frac{\norm{L^{\ast}ULy}^2}{\norm{ULy}^2}
\Frac{(f\circ L)(y)}{\norm{L^{\ast}ULy}^2}L^{\ast}ULy\nonumber\\
&=L^{\ast}Ly
-\Frac{\norm{L^{\ast}ULy}^2}{\norm{ULy}^2}\left(y-G_{f\circ L}
^{L^{\ast}UL}y\right).
\end{align}
NB: Comme $L^{\ast}ULy\neq0$, on a aussi $ULy\neq 0$.
\item
On a $(\forall x\in\HH)~\norm{L^\ast x}^2=
\scal{x}{LL^\ast x}=\alpha\norm{x}^2$, d'o\`u le r\'esultat.
\item
C'est le cas $\alpha=1$ du point pr\'ec\'edent.
\end{enumerate}
\end{proof}

\begin{notation}[Op\'erateur proximal]
\dueto{Definition 12.23}{Livre1}
Soient $f\in\Gamma_0(\HH)$. Pour tout $x$ dans $\HH$ on note 
$\prox_fx$ l'unique point de $\HH$ tel que
\begin{equation}
\underset{y\in\HH}{\min}\left(f(y)+\Frac{1}{2}\norm{x-y}^2\right)
=f(\prox_fx)+\Frac{1}{2}\norm{x-\prox_fx}^2
\end{equation} 
(voir \dueto{Proposition 12.15}{Livre1}).
\end{notation}

\begin{remark}
\dueto{Proposition 23.32}{Livre1}
S'il existe $\alpha\in\RR_{++}$ tel que $L^\ast L=LL^\ast=\alpha\Id$ 
alors
\begin{align}
L^\ast\circ\prox_{\alpha f}\circ L&=\alpha\prox_{f\circ L}\\
L^\ast\circ G_{\alpha f}^{\alpha U}\circ L&
=\alpha G_{f\circ L}^{L^\ast UL}
\end{align}
d'apr\`es la Proposition \ref{proj produit externe}.
\end{remark}

\begin{proposition}
Soient $\GG$ un espace de Hilbert sur $\RR$ et $L\in\BL(\HH,\GG)$ 
inversible avec $L^{-1}=L^\ast$. Posons
\begin{equation}
f:\HH\rightarrow\RR:x\mapsto\Frac{\norm{Lx}^2}{2}.
\end{equation}
Alors $G_f=\Frac{\Id}{2}=\prox_f$.
\end{proposition}

\begin{proof}
En appliquant le Corollaire \ref{proj puissance norme}, la 
Proposition \ref{proj produit externe} et la Proposition 
\ref{proj composition droite} on obtient $G_f=\Frac{\Id}{2}$. 
L'\'egalit\'e $\prox_f=\Frac{\Id}{2}$ r\'esulte de 
\dueto{Example 17.7}{Livre1} en prenant $r=0$ et $\gamma=1$.
\end{proof}

\subsubsection{Application vers l'acc\'el\'eration de l'algorithme 
de projection sous-diff\'erentielle}

\begin{lemma}
\label{prof composition gauche}
Soient $f\in\Gamma_0(\HH)$ \`a valeurs r\'eelles telle que $\lev{0}f
\neq\emp$ et $\phi:\RR\rightarrow\RR$ strictement croissante sur 
$f(\HH)$ avec $\phi(0)=0$. On suppose que $f$ et $\phi$ sont 
Fr\'echet-diff\'erentiables sur $\clev{0}f$. Alors
\begin{equation}
\left(\forall x\in\clev{0}f\right)\quad\norm{x-G_fx}-
\norm{x-G_{\phi\circ f}x}=\Frac{f(x)}{\norm{\nabla f(x)}}
\left(1-\Frac{\phi(f(x))}{f(x)\phi'(f(x)}\right).
\end{equation}
\end{lemma}

\begin{proof}
Il suffit d'appliquer la Proposition \ref{proj composition gauche}.
\end{proof}

On observe que $\lev{0}(\phi\circ f)=\lev{0}f$. Les op\'erateurs 
$G_{\phi\circ f}$ et $G_f$ sont donc tous deux des approximations 
sous-diff\'erentielles 
du projecteur m\'etrique sur $\lev{0}f$. Le 
r\'esultat suivant affirme que le projecteur de la composition \`a 
gauche par $\phi=|\cdot|^\alpha$ envoie le point plus loin, mais 
l'\'eloigne de sa projection m\'etrique.

\begin{proposition}
Soient $f\in\Gamma_0(\HH)$ \`a valeurs r\'eelles telle que $\lev{0}f
\neq\emp$ et $\alpha\in]0,1]$. On suppose $f\geq0$ et $f$ 
Fr\'echet-diff\'erentiable sur $\clev{0}f$. Alors
\begin{enumerate}
\item
$\left(\forall x\in\clev{0}f\right)\quad\norm{x-G_fx}-
\norm{x-G_{f^\alpha}x}=\Frac{f(x)}{\norm{\nabla f(x)}}
\left(1-\Frac{1}{\alpha}\right)\leq0$,
\item
$\left(\forall x\in\clev{0}f\right)\quad\norm{G_fx-P_{\lev{0}f}x}
\leq\norm{G_{f^\alpha}x-P_{\lev{0}f}x}$.
\end{enumerate}
\end{proposition}

\begin{proof}
\begin{enumerate}
\item
Il suffit d'appliquer le Lemme \ref{prof composition gauche} 
\`a $\phi:\RR\rightarrow\RR:t\mapsto|t|^\alpha$.
\item
Soit $x\in\clev{0}f$. Alors
\begin{align}
&\norm{G_fx-P_{\lev{0}f}x}^2-\norm{G_{f^\alpha}x-P_{\lev{0}f}x}^2
\nonumber\\
&=\norm{(x-P_{\lev{0}f}x)-\Frac{f(x)}{\norm{\nabla f(x)}^2}
\nabla f(x)}^2-\norm{(x-P_{\lev{0}f}x)-\alpha\Frac{f(x)}
{\norm{\nabla f(x)}^2}\nabla f(x)}^2\nonumber\\
&=-2(1-\alpha)\Frac{f(x)}{\norm{\nabla f(x)}^2}\scal{x-P_{\lev{0}f}x}
{\nabla f(x)}+(1-\alpha^2)\Frac{f(x)^2}{\norm{\nabla f(x)}^2}
\nonumber\\
&=(1-\alpha)\Frac{f(x)}{\norm{\nabla f(x)}^2}\left(2
\scal{P_{\lev{0}f}x-x}{\nabla f(x)}+(1+\alpha)f(x)\right)\nonumber\\
&\leq(1-\alpha)\Frac{f(x)}{\norm{\nabla f(x)}^2}
\left(2f(P_{\lev{0}f})+(\alpha-1)f(x)\right)\nonumber\\
&\leq-(1-\alpha)^2\Frac{f(x)^2}{\norm{\nabla f(x)}^2}\nonumber\\
&\leq0.
\end{align}
\end{enumerate}
\end{proof}

\subsection{Propri\'et\'es li\'ees \`a l'addition}

\begin{notation}
\'Etant donn\'ees deux fonctions $f$ et $g$ de $\Gamma_0(\HH)$, on 
notera  $\partial f\cap\partial g$ l'op\'erateur multivoque 
qui \`a tout point $x$ de $\HH$ associe l'ensemble $\partial f(x)
\cap\partial g(x)$.
\end{notation}

\begin{proposition}
\label{proj moyenne}
Soient $f$ et $g$ deux fonctions de $\Gamma_0(\HH)$ telles que 
$\lev{0}f\neq\emp$ et $\lev{0}g\neq\emp$,  et $\alpha\in]0,1[$. 
On note $h_{\alpha}=\alpha f+(1-\alpha)g$. On suppose 
$\lev{0}f\cap\lev{0}g\neq\emp$ et 
$\dom\left(\partial f\cap\partial g\right)\neq\emp$. Soit $U$ une 
s\'election de $\partial f\cap\partial g$.
\begin{enumerate}
\item
$G^U_{h_{\alpha}}$ est bien d\'efini.
\item
$\left(\dom G^U_f\right)\cap\left(\dom G_g^U\right)\subset\dom 
G_{h_{\alpha}}^U$
\item
Soit $x\in\left(\dom G^U_f\right)\cap\left(
\dom G^U_g\right)$. Alors
\begin{equation}
G_{h_{\alpha}}^U x=\alpha G^U_fx+(1-\alpha)
G^U_gx\Leftrightarrow f(x)g(x)\geq0.
\end{equation}
\end{enumerate}
\end{proposition}

\begin{proof}
\begin{enumerate}
\item
D'une part $\emp\neq(\lev{0}f)\cap(\lev{0}g)\subset
\lev{0}h_{\alpha}$, d'autre part $(\forall x\in\dom U)~Ux=
\alpha Ux+(1-\alpha)Ux\in\alpha\partial f(x)+(1-\alpha)\partial g(x)
\subset\partial h_{\alpha}(x)$.
\item
On a $\left(\dom G^U_f\right)\cap\left(\dom G^U_g\right)
=\left(\lev{0}f\cap\lev{0}g\right)\cup\dom U\subset\lev{0}h_{\alpha}
\cup\dom U=\dom G^U_{h_{\alpha}}$.
\item Soit $x\in\left(\dom G^U_f\right)\cap\left(
\dom G_g^U\right)$. Si $x\in\lev{0}f\cap\lev{0}g$ alors 
$G_{h_{\alpha}}^Ux=x=\alpha x+(1-\alpha)x=\alpha G^U_fx+(1-\alpha)
G^U_gx$. Supposons $x\in\complement(\lev{0}f\cap\lev{0}g)\cap
\dom U$. On a $G_f^Ux=x-\Frac{\left[f(x)\right]^+}{\norm{Ux}^2}Ux$, 
$G^U_gx=x-\Frac{\left[g(x)\right]^+}{\norm{Ux}^2}Ux$ et 
$G^U_{h_{\alpha}}x=x-\Frac{\left[\alpha f(x)+(1-\alpha)g(x)
\right]^+}{\norm{Ux}^2}Ux$. D'o\`u
\begin{equation}
G_{h_{\alpha}}^Ux-\alpha G^U_fx-(1-\alpha)G^U_gx
=\left(\alpha\left[f(x)\right]^++(1-\alpha)\left[g(x)\right]^+-
\left[\alpha f(x)+(1-\alpha)g(x)\right]^+\right)\Frac{Ux}{\norm{Ux}^2}
\end{equation}
Et ainsi
\begin{equation}
\label{expr projecteur moyenne}
G_{h_{\alpha}}^Ux-\alpha G^U_fx-(1-\alpha)G^U_gx=
\begin{cases}
(1-\alpha)\Frac{g(x)}{\norm{Ux}^2}Ux &
\text{si }f(x)\leq0,g(x)>0\text{ et }h_{\alpha}(x)\leq0\\
-\alpha\Frac{f(x)}{\norm{Ux}^2}Ux &
\text{si }f(x)\leq0,g(x)>0\text{ et }h_{\alpha}(x)>0\\
\alpha\Frac{f(x)}{\norm{Ux}^2}Ux & \text{si
}f(x)>0,g(x)\leq0\text{ et }h_{\alpha}(x)\leq 0\\
-(1-\alpha)\Frac{g(x)}{\norm{Ux}^2}Ux &
\text{si }f(x)>0,g(x)\leq0\text{ et }h_{\alpha}(x)>0\\
0 & \text{si }f(x)>0\text{ et }g(x)>0.
\end{cases}
\end{equation}
\end{enumerate}
\end{proof}

\begin{corollary}
\label{proj somme}
Soient $f$ et $g$ deux fonctions de $\Gamma_0(\HH)$ telles que 
$\lev{0}f\neq\emp$ et $\lev{0}g\neq\emp$. On suppose 
$\lev{0}f\cap\lev{0}g\neq\emp$ et 
$\dom\left(\partial f\cap\partial g\right)\neq\emp$. Soit $U$ une 
s\'election de $\partial f\cap\partial g$.
\begin{enumerate}
\item
$G_{\frac{f+g}{2}}^U$ et 
$G^{2U}_{f+g}$ sont bien d\'efinis.
\item
$\left(\dom G^U_f\right)\cap\left(\dom G_g^U
\right)\subset\dom G_{\frac{f+g}{2}}^U=\dom G^{2U}_{f+g}$
\item
Soit $x\in\left(\dom G^U_f\right)\cap\left(\dom
G_g^U\right)$. Alors
\begin{equation}
G^{2U}_{f+g}x=G_{\frac{f+g}{2}}^Ux=
\begin{cases}
\Frac{1}{2}G^U_fx+\Frac{1}{2}G^U_gx & \text{si }f(x)g(x)\geq0\\
\Frac{1}{2}G^U_fx+\Frac{1}{2}G^U_gx+\Frac{\min\left(|f(x)|,|g(x)|
\right)}{2\norm{Ux}^2}Ux & \text{sinon.}
\end{cases}
\end{equation}
\end{enumerate}
\end{corollary}

\begin{proof}
Posons $\tilde{f}=2f$, $\tilde{g}=2g$ et $\tilde{U}=2U$. 
On observe que $\lev{0}\tilde{f}\cap\lev{0}\tilde{g}\neq\emp$ et 
que $\tilde{U}$ est une s\'election de $\partial\tilde{f}\cap
\partial\tilde{g}$. Il suffit alors d'appliquer la 
Proposition \ref{proj moyenne} \`a $\tilde{f}$ et $\tilde{g}$.
\begin{enumerate}
\item
$G_{\frac{f+g}{2}}^U$ et $G^{\tilde{U}}_{\frac{\tilde{f}+
\tilde{g}}{2}} = G^{2U}_{f+g}$ sont bien d\'efinis.
\item
$\dom G^{\tilde{U}}_{\frac{\tilde{f}+\tilde{g}}{2}}=
\dom G^{2U}_{f+g}=\dom G_{\frac{f+g}{2}}^U$
\item
On a $G_{f+g}^{2U}x=G_{\frac{\tilde{f}+\tilde{g}}{2}}
^{\tilde{U}}x=\Frac{1}{2}G_{\tilde{f}}^{\tilde{U}}x+
\Frac{1}{2}G_{\tilde{g}}^{\tilde{U}}x=\Frac{1}{2}
G_{2f}^{2U}x+\Frac{1}{2}G_{2g}^{2U}x=\Frac{1}{2}G_f^Ux
+\Frac{1}{2}G_g^Ux$. De plus $G_{\frac{f+g}{2}}^Ux
=G_{2\frac{f+g}{2}}^{2U}x
=G_{f+g}^{2U}x$. Enfin, on d\'eduit l'expression de 
$G_{\frac{f+g}{2}}^Ux$ de (\ref{expr projecteur moyenne}) pour 
$\alpha=1/2$.
\end{enumerate}
\end{proof}

\begin{example}
Soient $(r,r')\in\RR_{++}^2$ tels que $r<r'$. Notons $B$ la boule 
unit\'e ferm\'ee de $\HH$. Alors l'op\'erateur $U$ d\'efini 
ci-dessous sur $\dom U=rB\cup\complement(r'B)$ est une s\'election 
de $\partial d_{rB}\cap\partial d_{r'B}$ de domaine maximal:
\begin{equation}
U:rB\cup\complement(r'B)\rightarrow\HH:x\mapsto
\begin{cases}
0 & \text{si }\norm{x}\leq r\\
\norm{x}^{-1}x & \text{si }r'\leq\norm{x}.
\end{cases}
\end{equation}
Soient $\alpha\in]0,1[$ et $x\in rB\cup\complement(r'B)$. Alors
\begin{align}
G_{{\displaystyle\alpha d_{rB}+(1-\alpha)d_{r'B}}}^Ux
&=\Frac{\alpha r+(1-\alpha)r'}{\norm{x}}x\\
G_{{\displaystyle d_{rB}+d_{r'B}}}^{2U}x
&=\Frac{r+r'}{2\norm{x}}x.
\end{align}
\end{example}

\begin{proof}
D'apr\`es \dueto{Proposition 18.22}{Livre1} on a
\begin{equation}
(\forall x\in\HH)\quad\partial d_{rB}(x)=
\begin{cases}
\{0\} & \text{si }\norm{x}<r\\
[0,r^{-1}x] & \text{si }\norm{x}=r\\
\left\{\norm{x}^{-1}x\right\} & \text{si }\norm{x}>r.
\end{cases}
\end{equation}
Par cons\'equent
\begin{equation}
(\forall x\in\HH)\quad\left(\partial d_{rB}\cap\partial 
d_{r'B}\right)(x)=
\begin{cases}
\{0\} & \text{si }\norm{x}\leq r\\
\emp & \text{si }r<\norm{x}<r'\\
\{\norm{x}^{-1}x\} & \text{si }r'\leq\norm{x}.
\end{cases}
\end{equation}
L'op\'erateur $U$ est donc bien une s\'election de 
$\partial d_{rB}\cap\partial 
d_{r'B}$ de domaine maximal. Par ailleurs, on a
\begin{equation}
\begin{cases}
\dom G_{{\displaystyle d_{rB}}}^U=rB\cup\dom U=\dom U\\
\dom G_{{\displaystyle d_{r'B}}}^U=r'B\cup\dom U=\HH\\
\end{cases}.
\end{equation}
D'o\`u
\begin{equation}
\left(\dom G_{{\displaystyle d_{rB}}}^U\right)
\cap\left(\dom G_{{\displaystyle d_{r'B}}}^U\right)=\dom U=rB\cup
\complement(r'B).
\end{equation}
Le reste d\'ecoule alors de la Proposition 
\ref{proj moyenne}, du Corollaire \ref{proj somme} et du fait que
\begin{equation}
(\forall x\in\HH)\quad G_{d_{rB}}x=P_{rB}x=
\begin{cases}
x & \text{si }\norm{x}\leq r\\
\Frac{r}{\norm{x}}x & \text{si }\norm{x}>r.
\end{cases}
\end{equation}
\end{proof}

\begin{notation}
\dueto{Definition 12.1}{Livre1}
Soient $f$ et $g$ deux fonctions de $\Gamma_0(\HH)$. Alors
\begin{equation}
f\infconv g:\HH\rightarrow\RXX:x\rightarrow
\underset{y\in\HH}{\inf}\left(f(y)+g(x-y)\right).
\end{equation}
Soit $x\in\HH$. On notera $(f\einfconv g)(x)$ si l'infimum est 
atteint.
\end{notation}

\begin{proposition}
\dueto{Proposition 12.6(ii)}{Livre1}
Soient $f$ et $g$ deux fonctions de $\Gamma_0(\HH)$. Alors 
$\dom f\infconv g=\dom f+\dom g$.
\end{proposition}

\begin{proposition}
\dueto{Proposition 16.48}{Livre1}
\label{sous-diff infconv exacte}
Soient $f$ et $g$ deux fonctions de $\Gamma_0(\HH)$, 
$x\in\dom(f\infconv g)$ et $y\in\HH$. Si 
$(f\infconv g)(x)=f(y)+g(x-y)$ alors $\partial(f\infconv g)(x)
=\partial f(y)\cap\partial g(x-y)$.
\end{proposition}

\begin{proposition}
\dueto{Proposition 12.14}{Livre1}
\label{infconv gamma}
Soient $f$ et $g$ deux fonctions de $\Gamma_0(\HH)$. On suppose 
qu'au moins une des deux assertions suivantes est vraie:
\begin{enumerate}
\item
$f$ est supercoercive,
\item
$f$ est coercive et $g$ est minor\'ee.
\end{enumerate}
Alors $f\infconv g=f\einfconv g\in\Gamma_0(\HH)$.
\end{proposition}

\begin{proposition}
\label{projecteur infconvolution}
Soient $f$ et $g$ deux fonctions de $\Gamma_0(\HH)$ telles que 
$\lev{0}f\neq\emp$ et $\lev{0}g\neq\emp$. On suppose $f\infconv g
=f\einfconv g\in\Gamma_0(\HH)$. Soit $M$ une s\'election de 
$x\mapsto\Argmin f+g(x-\cdot)$, de sorte que $(\forall x\in\HH)~
(f\einfconv g)(x)=f(Mx)+g(x-Mx)$. On suppose $\lev{0}(f\circ M)\cap
\lev{0}(g\circ(\Id-M))\neq\emp$ et 
$\dom(\partial f\cap\partial g)\neq\emp$. On suppose que $U$ est 
une s\'election de 
$\partial f\cap\partial g$ telle que $\im M\cup\im(\Id-M)\subset
\dom U$ et $(\forall x\in\dom U)~Ux=UMx=U(x-Mx)$.
\begin{enumerate}
\item
$G_{f\einfconv g}^U$ est bien d\'efini.
\item
\begin{multline*}
(\lev{0}(f\circ M)\cap\lev{0}(g\circ(\Id-M))\cup\dom U\subset\\
\dom\left(G_f^U\circ M\right)\cap\dom\left(G_g^U\circ(\Id-M)\right)
\cap\dom\left(G_{f\einfconv g}^U\right)
\end{multline*}
\item
Soit $x\in(\lev{0}(f\circ M)\cap\lev{0}(g\circ(\Id-M))\cup\dom U$. 
Alors
\begin{equation}
G_{f\einfconv g}^Ux=G_f^UMx+G_g^U(x-Mx)
\Leftrightarrow f(Mx)g(x-Mx)\geq0.
\end{equation}
\end{enumerate}
\end{proposition}

\begin{proof}
Notons $C=\lev{0}(f\circ M)$ et $D=\lev{0}(g\circ(\Id-M))$, on 
rappelle que $C\cap D\neq\emp$ par hypoth\`ese.
\begin{enumerate}
\item
D'une part $\lev{0}(f\einfconv g)\supset C\cap D\neq\emp$
, d'autre part
\begin{equation}
(\forall x\in\dom U)~Ux=UMx=U(x-Mx)\in\partial
f(Mx)\cap\partial g(x-Mx)=\partial(f\einfconv g)(x)
\end{equation}
d'apr\`es la Proposition \ref{sous-diff infconv exacte}.
\item
On a
\begin{equation}
\begin{cases}
\dom\left(G_f^U\circ M\right)=C\cup M^{-1}\left(\dom U\right)\\
\dom\left(G_g^U\circ(\Id-M)\right)=D\cup(\Id-M)^{-1}(\dom U)\\
\dom\left(G_{f\einfconv g}^U\right)=\lev{0}(f\einfconv g)\cup\dom U
\end{cases}
\end{equation} et 
\begin{equation}
\begin{cases}
\dom U\subset M^{-1}(\dom U)\\
\dom U\subset (\Id-M)^{-1}(\dom U)\\
C\cap D\subset\lev{0}(f\einfconv g).
\end{cases}
\end{equation}
Par cons\'equent $(C\cap D)\cup\dom U\subset\dom
\left(G_f^U\circ M\right)\cap\dom\left(G_g^U\circ(\Id-M)\right)\cap
\dom\left(G_{f\einfconv g}^U\right)$.
\item
Notons $y=Mx$. Si $f(Mx)\leq0$ et $g(x-Mx)\leq0$ alors 
$f\einfconv g(x)=f(y)+g(x-y)\leq0$, d'o\`u $G_{f\einfconv g}^Ux=x
=y+(x-y)=G_f^Uy+G_g^U(x-y)$. Supposons $x\in\complement(C\cap D)
\cap\dom U$ et posons $u=Ux$. Alors $\begin{cases}
G_{f\einfconv g}^Ux=x-\Frac{\left[f(y)+g(x-y)\right]^+}
{\norm{u}^2}u\\
G_f^Uy=y-\Frac{\left[f(y)\right]^+}{\norm{u}^2}u\\
G_g^U(x-y)=x-y-\Frac{\left[g(x-y)\right]^+}{\norm{u}^2}u
\end{cases}$. D'o\`u
\begin{equation}
G_{f\infconv g}^Ux-G_f^U y-G_g^U(x-y)=(\left[f(y)\right]^+
+\left[g(x-y)\right]^+-\left[f(y)+g(x-y)\right]^+)\Frac{u}
{\norm{u}^2}.
\end{equation}
Ainsi
\begin{equation}
G_{f\infconv g}^Ux-G_f^Uy-G_g^U(x-y)=
\begin{cases}
\Frac{g(x-y)}{\norm{u}^2}u & \text{si }f(y)\leq0,
g(x-y)>0\text{ et }(f\einfconv g)(x)\leq0\\
-\Frac{f(y)}{\norm{u}^2}u & \text{si }f(y)\leq0,
g(x-y)>0\text{ et }(f\einfconv g)(x)>0\\
\Frac{f(y)}{\norm{u}^2}u & \text{si }f(y)>0,g(x-y)
\leq0\text{ et }(f\einfconv g)(x)\leq0\\
-\Frac{g(x-y)}{\norm{u}^2}u & \text{si }f(y)>0,
g(x-y)\leq0\text{ et }(f\einfconv g)(x)>0\\
0 & \text{si }f(y)>0\text{ et }g(x-y)>0.
\end{cases}
\end{equation}
\end{enumerate}
\end{proof}

\subsection{Propri\'et\'es li\'ees \`a l'enveloppe de Moreau}

\begin{notation}[Enveloppe de Moreau]
\dueto{Definition 12.20}{Livre1}
Soient $f\in\Gamma_0(\HH)$ et $\gamma\in\RR_{++}$. Alors
\begin{equation}
^{\gamma}f=f\infconv\left(\Frac{1}{2\gamma}\norm{\cdot}^2\right).
\end{equation}
\end{notation}

\begin{proposition}
\dueto{Proposition 12.29}{Livre1}
\label{diff Moreau}
Soient $f\in\Gamma_0(\HH)$ et $\gamma\in
\RR_{++}$. $^{\gamma}f$ est Fr\'echet-diff\'erentiable sur $\HH$ et
\begin{equation}
\nabla(^{\gamma}f)=\gamma^{-1}(\Id-\prox_{\gamma f}).
\end{equation}
\end{proposition}

\begin{proposition}
Soient $f\in\Gamma_0(\HH)$ et $\gamma\in
\RR_{++}$ tels que $\lev{0}(^{\gamma}f)\neq\emp$. 
Alors $\dom G_{^{\gamma}f}=\HH$, $G_{^{\gamma}f}$ 
est univoque et
\begin{equation}
G_{^{\gamma}f}x=
\begin{cases}
x & \text{si }^{\gamma}f(x)\leq0\\
x-\gamma\Frac{^{\gamma}f(x)}{\norm{x-\prox_{\gamma f}x}^2}
(x-\prox_{\gamma f}x) & \text{si }^{\gamma}f(x)>0
\end{cases}
\end{equation}
\end{proposition}

\begin{remark}
Soit $C\subset\HH$ convexe ferm\'e non vide. 
En posant $f=\iota_C$ et $\gamma=1$ on retrouve l'Exemple 
\ref{proj distance carree} (d'apr\`es 
\dueto{Example 12.21}{Livre1}), \textit{i.e.} 
$G_{d_C^2/2}=G_{d_C^2}=\Frac{\Id+P_C}{2}$ (o\`u l'on a appliqu\'e 
la Proposition \ref{proj produit externe}).
\end{remark}

%
%

\section{R\'egularit\'e des s\'elections de $G_f$}

Int\'eressons-nous au propri\'et\'es de r\'egularit\'e des 
s\'elections du projecteur sous-diff\'erentiel, \`a savoir leur 
continuit\'e, leur diff\'erentiabilit\'e et leur caract\`ere 
lipschitzien.

\subsection{Continuit\'e des s\'elections de $G_f$}

\begin{proposition}
\dueto{Proposition 17.32}{Livre1}
\label{Frechet et selection}
Soit $f\in\Gamma_0(\HH)$ telle que $\lev{0}f\neq\emp$. On suppose 
$x\in\intdom f$. Les assertions suivantes sont \'equivalentes.
\begin{enumerate}
\item Il existe une s\'election $U$ de $\partial f$, telle que $x\in
\inte\dom U$, continue en $x$.
\item $f$ est Fr\'echet-diff\'erentiable en $x$.
\item Toute s\'election $U$ de $\partial f$ telle que $x\in
\inte\dom U$ est continue en $x$.
\end{enumerate}
\end{proposition}

\begin{proposition}
Soient $f\in\Gamma_0(\HH)$ telle que $\lev{0}f\neq\emp$ et $U$ une 
s\'election de $\partial f$. On suppose $x\in\clev{0}f\cap
\inte\dom U$. Les assertions suivantes sont \'equivalentes.
\begin{enumerate}
\item $G_f^U$ est continu en $x$.
\item $U$ est continue en $x$.
\item $f$ est Fr\'echet-diff\'erentiable en $x$.
\end{enumerate}
\end{proposition}

\begin{proof}
D'apr\`es \dueto{Corollary 8.30}{Livre1}, $\clev{0}f\cap 
\inte\dom U\subset\inte\dom f=\cont f$, donc $f$ est continue en 
$x$. On rappelle que, sur $\clev{0}f\cap\dom U$,
\begin{equation}
G_f^U=\Id-\Frac{f}{\norm{U}^2}U.
\end{equation}
Si $f$ est Fr\'echet-diff\'erentiable en $x$ alors la continuit\'e 
de $G_f^U$ en $x$ d\'ecoule directement de la Proposition 
\ref{Frechet et selection} et de l'\'egalit\'e ci-dessus. 
R\'eciproquement, supposons $G_f^U$ continu en $x$. Alors $\Frac{U}
{\norm{U}^2}=\Frac{\Id-G_f^U}{f}$ 
est continu en $x$ et $\norm{U}^2=\norm{\cdot}^{-2}\circ
\Frac{U}{\norm{U}^2}$ aussi. D'o\`u la continuit\'e en $x$ de 
$U=\Frac{\norm{U}^2}{f}\left(\Id-G^U_f\right)$. On en d\'eduit, 
via la Proposition \ref{Frechet et selection}, que $f$ est 
Fr\'echet-diff\'erentiable en $x$.
\end{proof}

\begin{corollary}
Soient $f\in\Gamma_0(\HH)$ telle que $\lev{0}f\neq\emp$ et $U$ une 
s\'election de $\partial f$. Les assertions suivantes 
sont \'equivalentes.
\begin{enumerate}
\item $G_f^U$ est continu sur $(\inte\lev{0}f)\cup(\inte\dom U)$.
\item $f$ est Fr\'echet-diff\'erentiable sur $\clev{0}f
\cap\inte\dom U$.
\end{enumerate}
Dans ce cas
\begin{equation}
G_f^U:\dom G_f^U\rightarrow2^\HH:x\mapsto
\begin{cases}
x & \text{si }f(x)\leq0\\
x-\Frac{f(x)}{\norm{\nabla f(x)}^2}\nabla f(x) & \text{si }f(x)>0
\text{ et }x\in\inte\dom U\\
x-\Frac{f(x)}{\norm{Ux}^2}Ux & \text{sinon.}
\end{cases}
\end{equation}
\end{corollary}

\subsection{Diff\'erentiabilit\'e et caract\`ere lispchitzien des 
s\'elections de $G_f$}

\subsubsection{Diff\'erentiabilit\'e}

\begin{lemma}
\label{diff inv}
$\norm{\cdot}^{-2}\Id$ est Fr\'echet-diff\'erentiable 
sur $\HH\backslash\{0\}$ et
\begin{equation}
(\forall x\in\HH\backslash\{0\})~\Diff(\norm{\cdot}^{-2}
\Id)(x)=\norm{x}^{-2}\Id-\norm{x}^{-4}\scal{2x}{\cdot}x
\end{equation}
\end{lemma}

\begin{proof}
Calcul effectu\'e gr\^ace \`a \dueto{X. Fonctions 
diff\'erentiables}{SR08}.
\end{proof}

\begin{remark}
Dans le cas r\'eel on retrouve la d\'eriv\'ee de la fonction 
inverse $\phi:\RR\backslash\{0\}\rightarrow\RR\backslash\{0\}:
t\mapsto\Frac{1}{t}$, qui n'est autre que $\phi':\RR\backslash\{0\}
\rightarrow\RR\backslash\{0\}:t\mapsto-\Frac{1}{t^2}$.
\end{remark}

\begin{lemma}
\label{gradient ordre deux}
Soit $f\in\Gamma_0(\HH)$. 
On suppose $x\in\inte\dom f$. On suppose $f$ 
Fr\'echet-diff\'erentiable au voisinage de $x$. Les assertions 
suivantes sont \'equivalentes.
\begin{enumerate}
\item
$\nabla f$ est Fr\'echet-diff\'erentiable en $x$.
\item
$f$ est deux fois Fr\'echet-diff\'erentiable en $x$.
\end{enumerate}
Dans ce cas
\begin{equation}
\Diff(\nabla f)(x)=\nabla^2f(x).
\end{equation}
\end{lemma}


\begin{proposition}
\label{diff proj}
Soit $f\in\Gamma_0(\HH)$ telle que $\lev{0}f\neq\emp$. 
On suppose $x\in\clev{0}f$. On suppose que $f$ est deux fois 
Fr\'echet-diff\'erentiable au voisinage de $x$. Alors 
\begin{equation}
\Diff G_f(x)=\Id-\Frac{\scal{\cdot}{\nabla f(x)}}
{\norm{\nabla f(x)}^2}\nabla f(x)-\Frac{f(x)}{\norm{\nabla f(x)}^2}
\nabla^2f(x)(\cdot)+\Frac{2f(x)\scal{\nabla f(x)}{\nabla^2f(x)
(\cdot)}}{\norm{\nabla f(x)}^4}\nabla f(x).
\end{equation}
\end{proposition}

\begin{proof}
Notons $\Inv=\norm{\cdot}^{-2}\Id$. Ainsi, sur $\clev{0}f\cap
\dom f$, $G_f=\Id-f\cdot\Inv\circ\nabla f$. On a $\Diff G_f(x)=
\Id-\scal{\cdot}{\nabla f(x)}(\Inv\circ\nabla f)(x)
-f(x)\cdot\Diff\Inv(\nabla f(x))\circ\Diff(\nabla f)(x)$ par 
diff\'erentiation d'une compos\'ee \dueto{X. Fonctions 
diff\'erentiables}{SR08}. Il ne reste alors plus qu'\`a appliquer 
le Lemme \ref{diff inv} :
\begin{equation}
\Diff G_f(x)=\Id-\Frac{\scal{\cdot}{\nabla f(x)}}
{\norm{\nabla f(x)}^2}\nabla f(x)-\Frac{f(x)}{\norm{\nabla f(x)}^2}
\nabla^2f(x)(\cdot)+\Frac{2f(x)\scal{\nabla f(x)}{\nabla^2f(x)
(\cdot)}}{\norm{\nabla f(x)}^4}\nabla f(x).
\end{equation}
\end{proof}

\begin{proposition}
\label{deriv proj}
Soit $f\in\Gamma_0(\RR)$. Si $f$ est deux fois diff\'erentiable 
alors $G_f$ est 
diff\'erentiable sur $\menge{x\in\RR}{f(x)\neq0}$ avec
\begin{equation}
(\forall x\in\complement\zer f)\quad G_f'x=
\begin{cases}
1 & \text{si }f(x)<0\\
\Frac{f''(x)f(x)}{f'(x)^2} & \text{si }f(x)>0.
\end{cases}
\end{equation}
\end{proposition}

\begin{remark}
Il ne suffit pas que $f$ soit deux fois diff\'erentiable pour que 
$G_f$ soit diff\'erentiable sur $\menge{x\in\RR}{f(x)=0}$.
\end{remark}

\begin{proof}
Il suffit de consid\'erer la fonction $f:\RR\rightarrow\RR:
x\mapsto x$. Dans ce cas
\begin{equation}
G_f:\RR\rightarrow\RR:x\mapsto
\begin{cases}
x & \text{si }x\leq0\\
0 & \text{si }x>0.
\end{cases}
\end{equation}
n'est pas d\'erivable en $0$.
\end{proof}

\subsubsection{Caract\`ere lipschitzien}

\begin{lemma}
Soient $f\in\Gamma_0(\HH)$ telle que $\lev{0}f\neq\emp$, $U$ une 
s\'election de $\partial f$ et $\beta\in\RR_{++}$. Les assertions 
suivantes sont \'equivalentes.
\begin{enumerate}
\item $G^U_f$ est $max(1,\beta)$-lipschitzien.
\item $\left(\forall(x,y)\in\left(\clev{0}f\cap\dom U\right)^2
\right)~\norm{x-y-\Frac{f(x)}{\norm{Ux}^2}Ux+\Frac{f(y)}{
\norm{Uy}^2}Uy}\leq\beta\norm{x-y}$
\end{enumerate}
\end{lemma}

\begin{proof}
Soit $(x,y)\in\left(\dom G^U_f\right)^2$. Si $f(x)\leq0$ et $f(y)
\leq0$ alors $\norm{G^U_fx-G^U_fy}=\norm{x-y}$. L'\'eventuelle 
constante de Lipschitz de $G^U_f$ est donc sup\'erieure ou \'egale 
\`a $1$. Supposons maintenant que $f(x)>0$ et $f(y)\leq0$. Alors
\begin{align}
\norm{G^U_fx-G^U_fy}^2-\norm{x-y}^2
&=\Frac{f(x)}{\norm{Ux}^2}
\left(f(x)+2\scal{y-x}{Ux}\right)\nonumber\\
&\leq\Frac{f(x)(2f(y)-f(x))}
{\norm{Ux}^2}\nonumber\\
&\leq0
\end{align}
On en d\'eduit imm\'ediatement le r\'esultat.
\end{proof}

\begin{proposition}
\dueto{V.5.1. Th\'eor\`eme}{SR08}
\label{jsr}
Soient $\Omega$ un ouvert convexe de $\HH$, $G:\Omega
\rightarrow\HH$ contin\^ument Fr\'echet-diff\'erentiable 
sur $\Omega$ et $\beta\in\RR_{++}$. Si $(\forall x\in\Omega)~
\norm{\Diff Gx}\leq\beta$ alors
\begin{equation}
\left(\forall(x,y)\in\Omega^2\right)~\norm{Gx-Gy}\leq\beta\norm{x-y}.
\end{equation}
\end{proposition}

\begin{proposition}
Soient $\Omega$ un ouvert convexe inclus dans $\clev{0}f\cap
\dom f$ et $\beta\in\RR_{++}$. On suppose $f$ deux fois 
Fr\'echet-diff\'erentiable sur $\Omega$. Si 
$\sup|f|(\Omega)<+\infty$, $\nabla f~\beta$-lipschitzien sur 
$\Omega$ et $\inf\norm{\nabla f(x)}(\Omega)>0$ alors $G_f$ est 
$\beta'$-lipschitzien sur $\lev{0}f\cup\Omega$ avec $\beta'=2+3
\Frac{\sup|f|(\Omega)}{\inf\norm{\nabla f(x)}^2(\Omega)}\beta$.
\end{proposition}

\begin{proof}
Gr\^ace \`a la Proposition \ref{diff proj} on trouve
\begin{equation}
\norm{\Diff G_f(x)}\leq2+3\Frac{|f(x)|\norm{\nabla^2 f(x)}}
{\norm{\nabla f(x)}^2}.
\end{equation}
Le r\'esultat est alors une cons\'equence directe de la 
Proposition \ref{jsr}.
\end{proof}

\begin{lemma}
\label{derivee bornee}
Soient $f:\RR\rightarrow\RR$ et $\beta\in\RR_{++}$. On suppose $f$ 
diff\'erentiable. Les assertions suivantes sont \'equivalentes.
\begin{enumerate}
\item
$f$ est $\beta$-lipschitzienne.
\item
$|f'|\leq\beta$
\end{enumerate}
\end{lemma}

\begin{proof}
Il s'agit d'une cons\'equence directe du th\'eor\`eme des 
accroissements finis \dueto{Th\'eor\`eme 4.1.}{SR08}.
\end{proof}

\begin{proposition}
Soient $f:\RR\rightarrow\RR$ et $\beta\in\RR_{++}$. 
On suppose $f$ deux fois diff\'erentiable. 
Si $\sup f(\clev{0}f)<+\infty$, 
$f'~\beta$-lispchitzienne et $\inf|f'|(\clev{0}f)>0$ 
alors $G_f$ est $\beta'$-lipschitzienne, o\`u 
$\beta'=\max\left\{1,\Frac{\sup f(\clev{0}f)}{\inf f'^2
(\clev{0}f)}\beta\right\}$.
\end{proposition}

\begin{proof}
Il suffit de combiner la Proposition \ref{deriv proj} et le Lemme 
\ref{derivee bornee}.
\end{proof}

\begin{remark}
Un projecteur sous-diff\'erentiel fermement contractant n'est pas 
n\'ecessairement un projecteur classique. En effet, si 
$C\subset\HH$ est un convexe ferm\'e non vide distinct de $\HH$, 
alors $G_{d_C^2}$ est une contraction ferme mais n'est pas un 
projecteur classique.
\end{remark}

\begin{proof}
D'apr\`es le Corollaire \ref{proj puissance distance}
\begin{equation}
G_{d_C^2}=\Frac{\Id+P_C}{2}
\end{equation}
$G_{d_C^2}$ est donc une contraction ferme puisque $P_C$ est 
contractant. Supposons par l'absurde que $G_{d_C^2}$ est un 
projecteur. Comme $\Fix G_{d_C^2}=\Fix P_C=C$, on en d\'eduit 
$G_{d_C^2}=P_C$, d'o\`u $P_C=\Id$, ce qui contredit l'hypoth\`ese 
$C\neq\HH$.
\end{proof}

\section{Propri\'et\'es s\'equentielles}

Voyons maintenant comment se comporte la suite des projecteurs 
sous-diff\'rentiels associ\'ee \`a une suite de fonctions de 
$\Gamma_0(\HH)$.

\subsection{Rappels sur les suites d'ensembles et notations}

\begin{notation}
Soit $(S_n)_{n\in\NN}$ une suite de parties de $\HH$.
\begin{equation}
\linf S_n=\displaystyle{\bigcup_{n\in\NN}\bigcap_{k\geq n}S_k}
\end{equation}
\begin{equation}
\lsup S_n=\displaystyle{\bigcap_{n\in\NN}\bigcup_{k\geq n}S_k}
\end{equation}
\end{notation}

\begin{definition}
\dueto{Definition 1.31}{Att84}
Soit $(S_n)_{n\in\NN}$ une suite de parties de $\HH$.
\begin{enumerate}
\item $\Li S_n=\displaystyle{\cl{\bigcup_{n\in\NN}\bigcap_{k\geq n}
\cl{S_k}}}$
\item $\Ls S_n=\displaystyle{\bigcap_{n\in\NN}\cl{\bigcup_{k\geq n}
S_k}}$
\end{enumerate}
\end{definition}

\begin{proposition}
\dueto{Proposition 1.33, Proposition 1.34}{Att84}
Soient $(S_n)_{n\in\NN}$ une suite de parties de $\HH$ et $x\in\HH$.
\begin{enumerate}
\item $x\in\Li S_n$ si et seulement si il existe une suite 
$(x_n)_{n\in\NN}$ de $\HH$ telle que $(\forall n\in\NN)~x_n\in S_n$ 
et $x_n\rightarrow x$.
\item$x\in\Ls S_n$ si et seulement si il existe une sous-suite 
$(S_{n_k})_{k\in\NN}$ et une suite $(x_k)_{k\in\NN}$ de $\HH$ 
telles que $(\forall k\in\NN)~x_k\in S_{n_k}$ et $x_k\rightarrow x$.
\item $\Li S_n\subset\Ls S_n$
\end{enumerate}
\end{proposition}

\begin{remark}
Soit $(S_n)_{n\in\NN}$ une suite de parties de $\HH$.
\begin{enumerate}
\item $\linf S_n\subset\Li S_n$
\item $\lsup S_n\subset\Ls S_n$
\end{enumerate}
\end{remark}

\subsection{Convergence forte}

\subsubsection{R\'esultats pr\'eliminaires li\'es \`a 
l'\'epi-convergence}

\begin{theorem}
\dueto{Theorem 3.66}{Att84}
Soient $(f_n)_{n\in\NN}$ une suite de fonctions de 
$\Gamma_0(\HH)$ 
telles que $(\forall n\in\NN)~\lev{0}f_n\neq\emp$ 
et $f\in\Gamma_0(\HH)$ telle que $\lev{0}f\neq\emp$. 
Les assertions suivantes sont \'equivalentes.
\begin{enumerate}
\item $\epi f_n\pk\epi f$
\item $\gra\partial f_n\pk\gra\partial f$ et
\begin{equation}
\left(\exi(x,u)\in
\gra\partial f\right)~\left(\exi((x_n,u_n))_{n\in\NN}\right)\in
\left(\HH^2\right)^\NN
\begin{cases}
(\forall n\in\NN)~(x_n,u_n)\in\gra\partial f_n\\
(x_n,u_n)\rightarrow(x,u)\\
f_n(x_n)\rightarrow f(x)
\end{cases}
\end{equation}
\end{enumerate}
\end{theorem}

\begin{lemma}
\label{epi1}
Soient $(f_n)_{n\in\NN}$ une suite de fonctions de 
$\Gamma_0(\HH)$ 
telles que $(\forall n\in\NN)~\lev{0}f_n\neq\emp$ 
et $f\in\Gamma_0(\HH)$ telle que $\lev{0}f\neq\emp$. 
Si $\Ls\epi f_n\subset\epi f$ alors $\Ls\left(\lev{0}f_n\right)
\subset\lev{0}f$.
\end{lemma}

\begin{proof}
Soit $x\in\Ls\left(\lev{0}f_n\right)$. Alors $(x,0)\in\Ls\epi f_n
\subset\epi f$, d'o\`u $f(x)\leq0$.
\end{proof}

\begin{lemma}
\label{epi2}
Soient $(f_n)_{n\in\NN}$ une suite de fonctions de 
$\Gamma_0(\HH)$ 
telles que $(\forall n\in\NN)~\lev{0}f_n\neq\emp$, 
$f\in\Gamma_0(\HH)$ telle que $\lev{0}f\neq\emp$ 
et $U$ une s\'election de $x\mapsto\Li\partial f_n(x)$. On 
suppose $\Li\left(\gra\partial f_n\right)\subset\gra\partial f$. 
Alors
\begin{enumerate}
\item $U$ est une s\'election de $\partial f$.
\item Il existe des s\'elections $U_n$ de $\partial f_n$, pour 
$n\in\NN$, d\'efinies sur $\dom U$ telles que $(\forall x\in\dom U)
~U_nx\rightarrow Ux$.
\end{enumerate}
\end{lemma}

\begin{proof}
Soit $x\in\dom U$. Il existe $(U_nx)_{n\in\NN}$ une suite 
convergente vers $Ux$ telle que $(\forall n\in\NN)~(x,U_nx)\in\gra
\partial f_n$. Ainsi $(x,Ux)\in\Li\left(\gra\partial f_n\right)
\subset\gra\partial f$.
\end{proof}

\subsubsection{R\'esultat de convergence forte}

\begin{lemma}
\label{inegalite}
Soient $f\in\Gamma_0(\HH)$ telle que $\lev{0}\neq\emp$ et $U$ une 
s\'election de $\partial f$.
\begin{equation}
\left(\forall x\in\clev{0}f\cap\dom U\right)~\left(
\forall y\in \lev{0}f\right)~\Frac{f(x)}{\norm{Ux}}\leq\norm{y-x}
\end{equation}
\end{lemma}

\begin{proof}
Soient $x\in\dom U$ tel que $f(x)>0$ et $y\in\lev{0}f$. Alors
\begin{align}
-\norm{y-x}\norm{Ux}+f(x)
&\leq\scal{y-x}{Ux}+f(x)\nonumber\\
&\leq f(y)\nonumber\\
&\leq0.
\end{align}
\end{proof}

\begin{remark}
Le lemme pr\'ec\'edent est une l\'eg\`ere extension du Corollaire 
\ref{prop heritees de T}\ref{ineg dist}.
\end{remark}

\begin{notation}
\'Etant donn\'ees $(f_n)_{n\in\NN}$ une suite de fonctions de 
$\Gamma_0(\HH)$ 
telles que $(\forall n\in\NN)~\lev{0}f_n\neq\emp$ et 
$U_n$ des s\'elections de $\partial f_n$ (pour $n\in\NN$), on 
notera $D$ l'ensemble des points $x$ de $\HH$ 
pour lesquels la suite de terme g\'en\'eral $G_{f_n}^{U_n}x$ est 
d\'efinie \`a partir d'un certain rang. Pour tout $x\in D$ on 
notera $N_x$ un tel rang (choisi arbitrairement).
\begin{equation}
D=\linf\dom G_{f_n}^{U_n}
\end{equation}
\end{notation}

\begin{lemma}
\label{result elem de conv}
Soient $(f_n)_{n\in\NN}$ une suite de fonctions de 
$\Gamma_0(\HH)$ 
telles que $(\forall n\in\NN)~\lev{0}f_n\neq\emp$,
$U_n$ des s\'elections de $\partial f_n$ (pour $n\in\NN$), 
$f\in\Gamma_0(\HH)$ telle que $\lev{0}f\neq\emp$, 
$U$ une s\'election de $\partial f$ d\'efinie sur $D$ et $x\in D$.
\begin{enumerate}
\item Si $f(x)\leq0$ et $x\in\Li\left(\lev{0}f_n\right)$ alors 
$G_{f_n}^{U_n}x\rightarrow G_f^Ux$.
\item Si $f(x)>0$ et $x\in\lsup\left(\lev{0}f_n\right)$ alors 
$G_{f_n}^{U_n}x\not\rightarrow G_f^Ux$.
\item Si $f(x)>0$, $x\in\complement\lsup\left(\lev{0}f_n\right)$, 
$f_n(x)\rightarrow f(x)$ et $U_nx\rightarrow Ux$ alors $G_{f_n}
^{U_n}x\rightarrow G_f^Ux$.
\end{enumerate}
\end{lemma}

\begin{proof}
Notons $G=G_f^U$ et, pour tout $n\in\NN$, $G_n=G_{f_n}^{U_n}$. 
Soit $N_x$ un entier tel que $(\forall n\geq N_x)~x\in\dom 
G_n$.
\begin{enumerate}
\item Soit $(x_n)_{n\in\NN}$ une suite de $\HH$ telle que $(
\forall n\in\NN)~f_n(x_n)\leq0$ et $x_n\rightarrow x$. On a
\begin{equation}
(\forall n\geq N_x)\quad\norm{G_nx-Gx}=
\begin{cases}
0 & \text{si }f_n(x)\leq0\\
\Frac{f_n(x)}{\norm{U_nx}} & \text{si }f_n(x)>0.
\end{cases}
\end{equation}
D'apr\`es le Lemme \ref{inegalite},
\begin{equation}
(\forall n\geq N_x)\quad\norm{G_nx-Gx}\leq\norm{x_n-x}.
\end{equation}
Par cons\'equent $G_nx\rightarrow Gx$.
\item Il existe une infinit\'e d'indices $n\in\NN$ tels que 
$f_n(x)\leq0$, donc tels que 
$\norm{G_nx-Gx}=\Frac{f(x)}{\norm{Ux}}>0$.
\item Soit $N_x'$ un entier sup\'erieur \`a $N_x$ tel que 
$(\forall n\geq N'_x)~x\in\clev{0}f_n\cap\dom U_n$. Alors
\begin{equation}
(\forall n\geq N_x')~G_nx-Gx=\Frac{f_n(x)}{\norm{U_nx}^2}U_nx-
\Frac{f(x)}{\norm{Ux}^2}Ux\rightarrow0.
\end{equation}
\end{enumerate}
\end{proof}

\begin{proposition}
Soient $(f_n)_{n\in\NN}$ une suite de fonctions de 
$\Gamma_0(\HH)$ 
telles que $(\forall n\in\NN)~\lev{0}f_n\neq\emp$, 
$f\in\Gamma_0(\HH)$ telle que $\lev{0}f\neq\emp$ et 
$U$ une s\'election de $x\mapsto\Li\partial f_n(x)$. On 
suppose $\epi f_n\pk\epi f$ et $\left(\forall x\in\clev{0}f\right)~
f_n(x)\rightarrow f(x)$. Alors
\begin{enumerate}
\item $U$ est une s\'election de $\partial f$ et $\dom U\subset D$,
\item Il existe des s\'elections $U_n$ de $\partial f_n$ (pour 
$n\in\NN$) d\'efinies sur $\dom U$ telles que
\begin{equation}
\forall x\in\Li\left(\lev{0}f_n\right)\cup\left(\clev{0}f
\cap\dom U)\right)~G_{f_n}^{U_n}x\rightarrow G_f^Ux
\end{equation}
\end{enumerate}
\end{proposition}

\begin{proof}
\begin{enumerate}
\item R\'esulte du Lemme \ref{epi2}.
\item R\'esulte de Lemme \ref{epi2} et Lemme 
\ref{result elem de conv}, en notant bien que 
$\clev{0}f\subset\complement\lsup\left(\lev{0}f_n\right)$ 
d'apr\`es le Lemme \ref{epi1}.
\end{enumerate}
\end{proof}

\section{Propri\'et\'es d'op\'erateur multivoque}

Pr\'esentons maintenant nos r\'esultats sur le projecteur 
sous-diff\'erentiel 
en tant qu'op\'erateur multivoque : les 
propri\'et\'es de ses valeurs et sa semi-continuit\'e.

\subsection{Propri\'et\'es des valeurs de $G_f$}

\begin{proposition}[Fonction inverse g\'en\'eralis\'ee]
\label{inv}
Soit $f\in\Gamma_0(\HH)$ telle que $\lev{0}f\neq\emp$. Utilisons 
la notation suivante:
\begin{equation}
\Inv:\HH\backslash\{0\}\rightarrow\HH\backslash\{0\}:x\mapsto
\norm{x}^{-2}x.
\end{equation}
\begin{enumerate}
\item
\label{expr avec inv}
Sur $\clev{0}f$, $G_f=\Id-f\Inv\circ\partial f$.
\item
\label{homeo}
$\Inv$ est un hom\'eomorphisme involutif de $\HH\backslash\{0\}$.
\item
$\Inv$ est faiblement continu si et seulement si $\dim\HH<+\infty$.
\item
\label{bornitude}
Soit $D\subset\HH\backslash\{0\}$. $\Inv(D)$ born\'e 
$\Longleftrightarrow0\in\complement\cl{D}$
\item $\Inv$ est Fr\'echet-diff\'erentiable sur 
$\HH\backslash\{0\}$ et
\begin{equation}
\left(\forall x\in\HH\backslash\{0\}\right)~\Diff(\Inv)(x)
=\norm{x}^{-2}\Id-2\scal{\Inv x}{\cdot}\Inv x.
\end{equation}
\end{enumerate}
\end{proposition}

\begin{proof}
\begin{enumerate}
\item D\'ecoule directement de la d\'efinition du projecteur 
sous-diff\'erentiel.
\item On v\'erifie facilement que $\Inv$ est continu et involutif, 
d'o\`u l'hom\'eomorphie.
\item
Si $\dim\HH<+\infty$, $\Inv$ est faiblement continu car continu. 
Supposons $\dim\HH=+\infty$. Soient $(e_i)_{i\in I}$ une base 
hilbertienne de $\HH$, $(e'_n)_{n\in\NN}$ une suite injective de 
$\menge{e_i}{i\in I}$ et $j\in I$. Notons, pour tout $n\in\NN$, 
$x_n=e'_n+e_j$. Alors $(\forall i\in I)~\scal{x_n-e_j}{e_i}=
\scal{e'_n}{e_i}\rightarrow0$ par injectivit\'e. On en d\'eduit, 
via \dueto{Proposition 2.40}{Livre1}, que $x_n\weakly e_j$. Pour 
tout $n\in\NN$, $\scal{\Inv x_n-\Inv e_j}{e_j}=
\Scal{\Frac{e'_n+e_j}{2(1+\scal{e'_n}{e_j})}-e_j}{e_j}=
-\frac{1}{2}\neq0$. Par cons\'equent, $\Inv x_n\not\weakly\Inv e_j$. 
$\Inv$ n'est donc pas faiblement continu.
\item
$\Inv(D)~\textup{ born\'e}\Leftrightarrow(\exi M\in\RR_{++})~
(\forall x\in D)~\norm{\Inv x}\leq M\Leftrightarrow(\exi M\in
\RR_{++})~(\forall x\in D)~M\leq\norm{x}\Leftrightarrow(\exi M'\in
\RR_{++})~\textup{B}(0;M)\subset\complement D\Leftrightarrow0\in
\inte\complement D=\complement\cl D$
\item Lemme \ref{diff inv}.
\end{enumerate}
\end{proof}

\begin{remark}
Dans le cas $\HH=\RR$, $\Inv$ n'est autre que la fonction inverse 
classique.
\end{remark}

\begin{corollary}
Soit $f\in\Gamma_0(\HH)$ telle que $\lev{0}f\neq\emp$. $G_f$ est 
\`a valeurs ferm\'ees born\'ees.
\end{corollary}

\begin{proof}
Soit $x\in\HH$. Si $x\in C$, $G_fx=\{x\}$ est ferm\'e born\'e 
(m\^eme compact). Supposons $x\in\complement C$. Comme $\partial 
f(x)$ est ferm\'e et ne contient pas $0$, on d\'eduit de 
Proposition \ref{inv}\ref{homeo} que 
$\Inv(\partial f(x))=\Inv^{-1}(\partial f(x))$ est ferm\'e et de 
Proposition \ref{inv}\ref{bornitude} que $\Inv(\partial f(x))$ est 
born\'e. Il d\'ecoule alors directement de Proposition \ref{inv}
\ref{expr avec inv} que $G_fx$ est ferm\'e born\'e.
\end{proof}

\begin{remark}
Les valeurs de $G_f$ ne sont pas n\'ecessairement convexes. 
Consid\'erons par exemple la fonction $f:\RR^2\rightarrow\RR:(x,y)
\mapsto1+\max(x,y)$. Alors $(u,v)\in\RR^2$ est un sous-gradient de 
$f$ en $(0,0)\in\clev{0}f$ si et seulement si
\begin{equation}
\label{carac exemple}
(\forall(x,y)\in\RR^2)~xu+yv\leq max(x,y).
\end{equation}
En appliquant \ref{carac exemple} aux valeurs $(-1,0),(1,0),(1,1)$ 
et $(-1,-1)$ de $(x,y)$ on obtient $u\in[0,1]$ et $v=1-u$. Ainsi 
$\menge{(u,1-u)}{u\in[0,1]}\subset\partial f(0,0)$. On v\'erifie 
facilement l'inclusion r\'eciproque gr\^ace \`a la 
caract\'erisation \ref{carac exemple}. On note que $\partial 
f(0,0)=[(1,0),(0,1)]$. $\Inv(\partial f(0,0))$ n'est pas convexe: 
cela s'observe facilement sur un dessin, ou bien en v\'erifiant 
par l'absurde que $\left(\Frac{1}{2},\Frac{1}{2}\right)\not\in
\Inv(\menge{(u,1-u)}{u\in[0,1]})$. Par cons\'equent $G_f(0,0)$ 
n'est pas convexe.
\end{remark}

\subsection{Semi-continuit\'e}

On notera $\Hfo$ l'espace $\HH$ muni de la topologie induite par 
la norme et $\Hfa$ l'espace $\HH$ muni de la topologie affaiblie.

\begin{definition}
\dueto{Definition 1.4.1, Definition 1.4.2, Definition 1.4.3}{AubFr90}
Soient $X$ et $Y$ deux espaces m\'etriques, $A:X\rightarrow2^Y$ et 
$x\in\dom A$.
\begin{itemize}
\item
On dira que $A$ est {\em semi-continu inf\'erieurement} en 
$x$ si et seulement si, pour toute suite $(x_n)_{n\in\NN}$ de 
$\dom A$ convergente vers $x$ et tout $y\in Ax$, il existe une 
suite $(y_n)_{n\in\NN}$ de $Y$ convergente vers $y$ telle que 
$(\forall n\in\NN)~y_n\in Ax_n$.
\item
On dira que $A$ est {\em semi-continu sup\'erieurement} en 
$x$ si et seulement si, pour tout voisinage $V$ de $Ax$, il existe 
un voisinage $U$ de $x$ tel que $A(U)\subset V$.
\item
On dira que $A$ est \em{continu} en $x$ si et seulement si $A$ est 
\`a la fois semi-continu sup\'erieurement en $x$ et semi-continu 
inf\'erieurement en $x$.
\end{itemize}
\end{definition}

\begin{remark}
Si $A$ est univoque au voisinage d'un point, les notions de 
semi-continuit\'e inf\'erieure, semi-continuit\'e sup\'erieure et 
continuit\'e en ce point se confondent.
\end{remark}

\subsubsection{Semi-continuit\'e du sous-diff\'erentiel}

\begin{proposition}
\label{sousdiff sci}
Soient $f\in\Gamma_0(\HH)$. Supposons $x\in\inte\dom f$.
Si $f$ est Fr\'echet-diff\'erentiable en $x$ alors $\partial f$ est 
semi-continu inf\'erieurement en $x$.
\end{proposition}

\begin{proof}
Soit $(x_n)_{n\in\NN}$ une suite de $\dom\partial f$ convergente 
vers $x$ et $u\in\partial f(x)$. Soit $U$ une s\'election de 
$\partial f$ d\'efinie sur $\dom\partial f$ telle que $Ux=u$. 
D'apr\`es la Proposition \ref{Frechet et selection} $U$ est 
continue en $x$, donc $Ux_n\rightarrow Ux=u$ et $(\forall n\in\NN)~
Ux_n\in\partial f(x_n)$.
\end{proof}

\begin{definition}
Soient $A:\HH\rightarrow2^\HH$ et $x_0\in\dom A$. On dira que $A$ 
est {\em h\'emi-continu sup\'erieurement} en $x_0$ si et seulement 
si, pour tout $u\in\HH$, la fonction $\HH\rightarrow\RXX:x\mapsto
\sigma_{Ax}(u)$ est semi-continue sup\'erieurement en $x_0$.
\end{definition}

\begin{theorem}
\dueto{3.3.Theorem 10}{AubEk84}
Soient $A:\Hfo\rightarrow2^{\Hfa}$ et $x_0\in\dom A$. On suppose 
$A$ h\'emi-continu sup\'erieurement $x_0$ et $Ax_0$ convexe 
faiblement compact. Alors $A$ est semi-continu sup\'erieurement en 
$x_0$.
\end{theorem}

\begin{theorem}
\dueto{4.3.Theorem 17}{AubEk84}
\label{sousdiff hcs}
Soit $f\in\Gamma_0(\HH)$. $\partial f$ est h\'emi-continu 
sup\'erieurement, \`a valeurs non vides convexes faiblement 
compactes, sur $\inte\dom f$.
\end{theorem}

\begin{corollary}
\label{sousdiff scs}
Soit $f\in\Gamma_0(\HH)$. $\partial f:\Hfo\rightarrow
2^{\Hfa}$ est semi-continu sup\'erieurement sur $\inte\dom f$.
\end{corollary}

\subsubsection{Semi-continuit\'e du projecteur 
sous-diff\'erentiel}

\begin{lemma}
\label{proj sci sur lev}
Soit $f\in\Gamma_0(\HH)$ telle que $\lev{0}f\neq\emp$. $G_f$ est 
continu sur $\inte\lev{0}f$ et semi-continu inf\'erieurement sur 
$\fr\lev{0}f$.
\end{lemma}

\begin{proof}
Soient $x\in\lev{0}f$ et $(x_n)_{n\in\NN}$ une suite de $\dom G_f$. 
Pour toute suite $(y_n)_{n\in\NN}$ de $\HH$ telle que $(\forall 
n\in\NN)~y_n\in G_fx_n$ on a
\begin{equation}
\left(\forall n\in\NN\right)~\norm{y_n-x}\leq\norm{x_n-x}
\rightarrow0
\end{equation}
par quasi-contractance de $G_f$ (voir Proposition 
\ref{quasi-contractance}). Si de plus $x\in\inte\lev{0}f$, 
alors, pour tout voisinage $V$ de $x$, $G_f\left(V\cap\inte\lev{0}f
\right)\subset V$.
\end{proof}

\begin{proposition}
Soit $f\in\Gamma_0(\HH)$ telle que $\lev{0}f\neq\emp$. Si $f$ est 
Fr\'echet-diff\'erentiable sur $\inte\dom f\cap\clev{0}f$ alors 
$G_f$ est continu sur $\inte\lev{0}f\cup(\inte\dom f\cap\clev{0}f)$ 
et semi-continu inf\'erieurement sur $\fr\lev{0}f$.
\end{proposition}

\begin{proof}
Par Fr\'echet-diff\'erentiabilit\'e $\partial f=\nabla f$ et $f$ 
sont continus sur $\inte\dom f\cap\clev{0}f$. Par cons\'equent 
$G_f=\Id-f\Inv\circ\nabla f$ est continu sur $\inte\dom f\cap
\clev{0}f$. Le reste d\'ecoule du Lemme \ref{proj sci sur lev}.
\end{proof}

\begin{proposition}
\dueto{Proposition 1.4.14}{AubFr90}
\label{operateur parametre}
Soient $X$, $Y$ et $Z$ des espaces m\'etriques, $A:X\rightarrow2^Z$ 
et $g:\gra A\rightarrow Y$. On d\'efinit l'op\'erateur suivant:
\begin{equation}
B:X\rightarrow2^Y:x\mapsto g\left(x,Ax\right).
\end{equation}
On suppose $g$ continue.
\begin{enumerate}
\item
Si $A$ est semi-continu inf\'erieurement alors $B$ aussi.
\item
Si $A$ est semi-continu sup\'erieurement 
\`a valeurs compactes alors $B$ aussi.
\end{enumerate}
\end{proposition}

\begin{corollary}
Soit $f\in\Gamma_0(\HH)$ telle que $\lev{0}f\neq\emp$. 
Si $\dim\HH<+\infty$ alors $G_f$ est semi-continu sup\'erieurement 
sur $\inte\dom f$.
\end{corollary}

\begin{proof}
D\'efinissons l'op\'erateur suivant:
\begin{equation}
A:\HH\rightarrow2^\HH:x\mapsto
\begin{cases}
\partial f(x) & \text{si }x\in\inte\dom f\\
\emp & \text{sinon.}
\end{cases}
\end{equation}
D'apr\`es le Th\'eor\`eme \ref{sousdiff hcs} on a $\dom A=
\inte\dom f$. On d\'efinit \'egalement la fonction suivante:
\begin{equation}
g:\gra A\rightarrow\HH:(x,u)\mapsto
\begin{cases}
x & \text{si }f(x)\leq0\\
x-f(x)\Inv u & \text{si }f(x)>0.
\end{cases}
\end{equation}
V\'erifions que $g$ est continue. Soit $(x,u)\in\gra A$ et 
$(x_n,u_n)_{n\in\NN}$ une suite de $\gra A$ convergente vers 
$(x,u)$. Supposons $f(x)\leq0$ et donnons-nous une s\'election $U$ 
de $\partial f$ telle que $Ux=u$ et $(\forall n\in\NN)~Ux_n=u_n$. 
Comme $G_f^U$ est une quasi-contraction (Proposition 
\ref{quasi-contractance}) on obtient:
\begin{equation}
(\forall n\in\NN)~\norm{g(x_n,u_n)-g(x,u)}=\norm{G_f^Ux_n-G_f^Ux}
\leq\norm{x_n-x}\rightarrow0.
\end{equation}
Supposons maintenant que $f(x)>0$. Comme $\clev{0}f$ est ouvert, 
pour $n$ assez grand on a $f(x_n)>0$ et, comme $\Inv$ est continu 
sur $\HH\backslash\{0\}$ et $x\in\inte\dom f=\cont f$ 
(\textit{c.f.} \dueto{Corollary 8.30}{Livre1}),
\begin{equation}
g(x_n,u_n)=x_n-f(x_n)\Inv u_n\rightarrow x-f(x)\Inv u=g(x,u).
\end{equation}
D\'efinissons l'op\'erateur $B$ comme dans la Proposition 
\ref{operateur parametre}:
\begin{equation}
B:\HH\rightarrow2^\HH:x\mapsto
\begin{cases}
G_fx & \text{si }x\in\inte\dom f\\
\emp & \text{sinon.}
\end{cases}
\end{equation}
D'apr\`es le Th\'eor\`eme \ref{sousdiff hcs} et son corollaire 
\ref{sousdiff scs} $A$ est semi-continu sup\'erieurement \`a 
valeurs compactes. La Proposition \ref{operateur parametre} permet 
de conclure que, sur $\inte\dom f=\dom B$, $B=G_f$ est semi-continu 
sup\'erieurement.
\end{proof}

\subsection{Monotonie}

\begin{lemma}
Soit $f\in\Gamma_0(\HH)$ telle que $\lev{0}f\neq\emp$. Les 
assertions suivantes sont \'equivalentes.
\begin{enumerate}
\item $G_f$ est monotone.
\item
\begin{multline*}
\left(\forall (x,y)\in\left(\clev{0}f\right)^2
\right)\left(\forall u\in\partial f(x)\right)\left(\forall v\in
\partial f(y)\right)\\
\scal{x-y}{f(x)\Inv u-f(y)\Inv v}
\leq\norm{x-y}^2
\end{multline*}
\end{enumerate}
\end{lemma}

\begin{proof}
Soit $(x,y)\in\left(\dom G_f\right)^2$. Si $f(x)\leq0$ et $f(y)
\leq0$ alors
\begin{equation}
\scal{G_fx-G_fy}{x-y}=\norm{x-y}^2\geq0.
\end{equation}
Supposons 
maintenant que $f(x)>0$ et $f(y)\leq0$, et donnons-nous $u\in
\partial f(x)$. Alors
\begin{equation}
\label{def monotonie}
\scal{G_fx-G_fy}{x-y}=\norm{x-y}^2-\Frac{f(x)}{\norm{u}^2}
\scal{u}{x-y}
\end{equation}
On a
\begin{align}
-\norm{u}\norm{y-x}+f(x)
&\leq\scal{u}{y-x}+f(x)\nonumber\\
&\leq f(y)\nonumber\\
&\leq0.
\end{align}
On en d\'eduit les relations suivantes:
\begin{eqnarray}
\scal{u}{x-y}\geq f(x)-f(y)>0\\
\Frac{f(x)}{\norm{u}}\leq\norm{y-x}.
\end{eqnarray}
Par cons\'equent la relation (\ref{def monotonie}) donne
\begin{align}
\scal{G_fx-G_fy}{x-y}
&\geq\norm{x-y}^2-\Frac{\norm{y-x}}{\norm{u}}
\scal{u}{x-y}\nonumber\\
&=\Frac{\norm{x-y}}{\norm{u}}\left(\norm{x-y}\norm{u}-
\scal{u}{x-y}\right)\nonumber\\
&\geq0.
\end{align}
L'op\'erateur $G_f$ est donc monotone si et seulement si la relation 
d\'efinissant la monotonie est v\'erifi\'ee pour les points de 
$\clev{0}f$.
\end{proof}

\begin{example}
\'Etant donn\'es un convexe ferm\'e non vide $C\subset\HH$ et 
$\alpha\in\left]0,1/2\right]\cup\{1\}$, on a vu au 
Corollaire \ref{proj puissance distance} que $G_{d_C^{1/\alpha}}$ 
\'etait une contraction ferme, donc monotone.
\end{example}

\newpage

\section*{Conclusion}

Dans ce m\'emoire nous avons \'etudi\'e la projection 
sous-diff\'erentielle 
dans le cadre de la th\'eorie des op\'erateurs, 
et non pas seulement comme m\'ethode algorithmique. \`A notre 
connaissance cette d\'emarche est nouvelle. Nous avons obtenu des 
propri\'et\'es alg\'ebriques li\'ees \`a la composition que 
d'aucuns pourraient qualifier d'\'el\'egantes, mais celles 
li\'ees \`a l'addition et \`a l'inf-convolution ne sont pas aussi 
applicables qu'esp\'er\'e. De m\^eme les interactions avec la 
conjugaison de Fenchel et l'enveloppe de Moreau se sont 
r\'ev\'el\'ees faibles. Nous avons \'egalement \'etabli un fort 
lien entre Fr\'echet-diff\'erentiabilit\'e et continuit\'e du 
projecteur sous-diff\'erentiel. Nous nous attendions cependant 
\`a trouver des conditions plus simples sous lesquelles $G_f$ 
soit une contraction, une contraction ferme, ou m\^eme un 
projecteur m\'etrique.


\bibliographystyle{plain}

\end{document}